\documentclass[11pt,twoside]{article}
\usepackage[margin=1.3in]{geometry}

\usepackage{latexsym,amsfonts,amsmath,amsthm,amssymb,makeidx}
\usepackage[title]{appendix}
\usepackage{cases,cite,enumitem,makeidx,times}
\usepackage{CJK,CJKnumb,CJKulem,times,dsfont,ifthen,mathrsfs,latexsym,amsfonts, color}
\usepackage{amsmath,amsthm,makeidx,fontenc,amssymb,bm,graphicx,psfrag,listings, curves,extarrows}
\usepackage[colorlinks,citecolor=black,linkcolor=black]{hyperref}

\let\oldbibliography\thebibliography
\renewcommand{\thebibliography}[1]{%
\oldbibliography{#1}%
\setlength{\itemsep}{0pt}%
}

\makeindex
\newtheorem{theorem}{Theorem}[section]
\newtheorem{lemma}[theorem]{Lemma} 
\newtheorem{proposition}[theorem]{Proposition}
\newtheorem{definition}[theorem]{Definition}
\newtheorem{example}[theorem]{Example}
\newtheorem{corollary}[theorem]{Corollary}
\newtheorem{remark}[theorem]{Remark}

\newcommand{\Rn}{\mathbb R^n}

\newcommand{\R}{\mathbb R}

\newcommand{\bt}{\begin{theorem}}
\newcommand{\et}{\end{theorem}}
\newcommand{\bl}{\begin{lemma}}
\newcommand{\el}{\end{lemma}}
\newcommand{\bd}{\begin{definition}}
\newcommand{\ed}{\end{definition}}
\newcommand{\bc}{\begin{corollary}}
\newcommand{\ec}{\end{corollary}}
\newcommand{\bp}{\begin{proof}}
\newcommand{\ep}{\end{proof}}
\newcommand{\bx}{\begin{example}}
\newcommand{\ex}{\end{example}}
\newcommand{\bi}{\begin{exercise}}
\newcommand{\ei}{\end{exercise}}
\newcommand{\bo}{\begin{prop}}
\newcommand{\eo}{\end{prop}}
\newcommand{\br}{\begin{remark}}
\newcommand{\er}{\end{remark}}
\newcommand{\be}{\begin{equation}}
\newcommand{\ee}{\end{equation}}
\newcommand{\ba}{\begin{align}}
\newcommand{\ea}{\end{align}}
\newcommand{\bn}{\begin{enumerate}}
\newcommand{\en}{\end{enumerate}}
\newcommand{\bg}{\begin{align*}}
\newcommand{\bcs}{\begin{cases}}
\newcommand{\ecs}{\end{cases}}

\newcommand{\bean}{\begin{eqnarray*}}
\newcommand{\eean}{\end{eqnarray*}}

\numberwithin{equation}{section}

\begin{document}
\title{ {\bf  Liouville-type  theorems,  radial symmetry and integral representation of solutions to Hardy-H\'enon equations involving higher order fractional Laplacians}   } 
\date{\today}     
\author{\\ { Hui Yang\thanks{School of Mathematical Sciences, Shanghai Jiao Tong University, Shanghai 200240, China. Email address: \textsf{hui-yang@sjtu.edu.cn}. } }  }

\maketitle    

\begin{center}
\begin{minipage}{130mm}
\begin{center}{\bf Abstract}\end{center}
We study nonnegative solutions to the following Hardy-H\'enon  type equations involving higher order fractional Laplacians  
$$
(-\Delta)^\sigma u = |x|^{-\alpha}u^{p} ~~~~~~  \textmd{in} ~ \mathbb{R}^n \backslash \{0\} 
$$
with a possible singularity at the origin, where $\sigma$ is a real number satisfying $0 < \sigma < n/2$, $-\infty < \alpha < 2\sigma$ and $p>1$.  By a more direct approach without using the super poly-harmonic properties, we establish an integral representation for nonnegative solutions to the above higher order fractional equations whether the singularity $\{0\}$ is removable or not. As the first application, we prove an optimal Liouville-type theorem for the above equations with removable singularity for all $\sigma \in (0, n/2)$ when 
$$
1 < p < p_{\sigma,\alpha}^*:=\frac{n+2\sigma -2\alpha}{n-2\sigma}~~~ \textmd{and} ~~~  -\infty < \alpha < 2\sigma. 
$$
This, in particular, covers a gap occurring for non-integral $\sigma  \in (1, n/2)$ and $\alpha \in (0,  2\sigma)$ in the current literature. 
As the second application, we show the radial symmetry of solutions in the critical case or in the case when the origin is a non-removable singularity.  Such radial symmetry would be useful in studying the singular Yamabe-type problems. 

\vskip0.10in


\vskip0.10in


\end{minipage}

\end{center}

\section{Introduction and main results} 
Problems involving the fractional Laplacian operator and other nonlocal diffusion operators have been extensively studied in recent years.  Such nonlocal problems arise naturally in mathematical models from various physical phenomena,  such as anomalous diffusion and quasi-geostrophic flows, turbulence and water waves, molecular dynamics, and relativistic quantum mechanics of stars (see \cite{BG,CV,Con} and the references therein). Moreover, the fractional Laplacian and higher order Laplacian also appear in many areas of mathematics,  including  conformal geometry \cite{CC,CG},  probability and finance \cite{App,ContT},  obstacle problems \cite{CSS,Silv}, nonlinear elasticity and crystal dislocation \cite{Ann,DFV}.

In this paper, we focus on the nonnegative solutions to the following Hardy-H\'enon  type equations involving higher order fractional Laplacians   
\begin{equation}\label{Har=01}
(-\Delta)^\sigma u = |x|^{-\alpha}u^{p} ~~~~~~  \textmd{in} ~ \mathbb{R}^n \backslash \{0\}
\end{equation}
which have a possible singularity at the origin, where $\sigma$ is a real number satisfying $0 < \sigma < n/2$, $-\infty < \alpha < 2\sigma$ and $p>1$.   The fractional Laplacian is defined, for $\sigma \in (0, 1)$,  by 
\begin{equation}\label{Fda-Lapcc=01}
(-\Delta)^\sigma u(x) = c_{n,\sigma}  \textmd{P.V.}  \int_{\mathbb{R}^n} \frac{u(x) - u(z)}{|x - z|^{n+2\sigma}} dz
\ee
while, for higher powers, say $\sigma=k + \sigma^\prime$ with $k \in \mathbb{N}$ and $\sigma^\prime \in (0, 1)$, 
\begin{equation}\label{Fda-Lapcc=02}
(-\Delta)^\sigma = (-\Delta)^{\sigma^\prime} \circ (-\Delta)^k. 
\ee

When $\sigma=1$, $0\leq  \alpha<2$ and $p=2^*(\alpha)-1$ with $2^*(\alpha):=\frac{2(n-\alpha)}{n-2}$,  Eq. \eqref{Har=01} arises as an Euler-Lagrange equation of a functional associated with the following Hardy-Sobolev inequality: There exists $C>0$ such that for all $u \in \mathcal{D}^{1,2} (\mathbb{R}^n)$, 
\begin{equation}\label{Har-Sob=01}
\left(\int_{\mathbb{R}^n} \frac{|u|^{2^*(\alpha)}}{|x|^\alpha} dx \right)^{\frac{2}{2^*(\alpha)}} \leq C \int_{\mathbb{R}^n} |\nabla u|^2 dx. 
\end{equation}
In particular, when $\alpha=0$, \eqref{Har-Sob=01} becomes the classical Sobolev inequality whose best constant and extremal functions have been identified by Aubin \cite{A} and Talenti \cite{T}.  Furthermore, the best constant in the Hardy--Sobolev inequality \eqref{Har-Sob=01} was first computed by Glaser-Martin-Grosse-Thirring \cite{GMGT} and extremal functions were identified by Lieb \cite{Lieb}.  For general $\sigma \in (0, n/2)$, Eq. \eqref{Har=01} with $0 \leq  \alpha < 2\sigma$  is also closely related to the higher order fractional Hardy-Sobolev inequality (see, e.g., \cite{MN,YJ}).  On the other hand,   Eq. \eqref{Har=01} has been used in models of astrophysics (see, e.g., \cite{H}) when $\sigma=1$ and $\alpha < 0$.

Notice also that Eq. \eqref{Har=01} with $\alpha=0$ and $p=(n+2\sigma)/(n-2\sigma)$ corresponds to the fractional Yamabe problem in conformal geometry \cite{CC,CG}. To be more specific, denote $g=u^{\frac{4}{n-2\sigma}} |dx|^2$ the conformal change of the Euclidean metric $|dx|^2$ for some positive function $u$. Then one can construct the fractional order conformal Laplacian associated with $g$, which satisfies 
\begin{equation}\label{Fda-Lapcc=03}
P_\sigma^g (\cdot) = u^{-\frac{n+2\sigma}{n-2\sigma}} (-\Delta)^\sigma (u \cdot). 
\ee
The fractional $Q$-curvature of $g$ is given by 
\begin{equation}\label{QQ=g03r}
Q_\sigma^g:=P_\sigma^g (1)=u^{-\frac{n+2\sigma}{n-2\sigma}} (-\Delta)^\sigma  u. 
\end{equation} 
Thus, a solution $u$ of \eqref{Har=01} with $\alpha=0$ and $p=(n+2\sigma)/(n-2\sigma)$ induces a conformally flat metric $g=u^{\frac{4}{n-2\sigma}} |dx|^2$ which has constant fractional $Q$-curvature $Q_\sigma^g=1$.  In particular, when $\sigma=1$, curvature \eqref{QQ=g03r} is the scalar curvature modulo a positive constant, while for $\sigma=2$, it coincides with the Branson's $Q$-curvature corresponding to the Paneitz operator.

Eq. \eqref{Har=01} that involves the higher order fractional Laplacians has both nonlocal features and the difficulty of higher order elliptic equations. When $\sigma=m \in (0, n/2)$ is an integer and $\alpha=0$, Wei and Xu \cite{WX} proved the following super poly-harmonic properties for positive classical solutions of  \eqref{Har=01} in $\mathbb{R}^n$
\be\label{Super-hHowk=01}
(-\Delta)^i u\geq 0 ~~~ \textmd{in} ~\mathbb{R}^n, ~~~ i=1, 2, \dots, m-1, 
\ee
and they further applied this information to classify the entire solutions of \eqref{Har=01} in $\mathbb{R}^n$ for $1 < p \leq \frac{n+2m}{n-2m}$.  See the same results by Lin \cite{Lin98} in the case of $m=2$.  Later,  Chen, Li and Ou in \cite{CLO06} used the super poly-harmonic properties \eqref{Super-hHowk=01} to establish an integral representation formula for positive entire solutions of \eqref{Har=01} in $\mathbb{R}^n$ and gave a new proof to the classification results of Lin \cite{Lin98} and Wei-Xu \cite{WX} via the corresponding integral equation. We may also see \cite{ChengLiu,DPQ,DQ,NY} for the similar super poly-harmonic properties and integral representation of nonnegative classical solutions to \eqref{Har=01} in $\mathbb{R}^n$  when $\sigma=m \in (0, n/2)$ is an integer and $-\infty < \alpha< 2m$.  Moreover, an integral representation for nonnegative entire solutions to \eqref{Har=01} on $\mathbb{R}^n$ in the fractional setting ($0< \sigma <1$) has also been obtained in \cite{BQ,DQ,ZCCY} recently to prove Liouville-type theorems.  Thus, a natural question is whether one can establish an integral representation for nonnegative solutions of \eqref{Har=01} when $\sigma \in (1, n/2)$ is not an integer. In a recent paper \cite{CDQ},  Cao, Dai and Qin showed the following super poly-harmonic properties
\be\label{Su-hHe8i=01}
(-\Delta)^{i+\frac{\nu}{2}} u\geq 0 ~~~ \textmd{in} ~\mathbb{R}^n, ~~~ i=0, 1, \dots, m-1 
\ee
for nonnegative classical solutions of \eqref{Har=01} in $\mathbb{R}^n$ with $\alpha \leq 0$,  where $\sigma$ is rewritten as $\sigma=m+\frac{\nu}{2}$ for some integer $m \geq 1$ and some $0< \nu <2$.  Based on the super poly-harmonic properties \eqref{Su-hHe8i=01}, they further proved an integral representation formula for nonnegative classical solutions of \eqref{Har=01} in $\mathbb{R}^n$ when $1\leq \sigma< n/2$ and $\alpha \leq 0$.    In fact, they studied higher order fractional equations with general continuous nonlinearities which include $|x|^{-\alpha} u^p$ ($\alpha \leq 0$) as a special case.

In this paper, by a more direct approach without using the super poly-harmonic properties \eqref{Su-hHe8i=01},  we will establish an integral representation of nonnegative solutions to the higher order fractional equations \eqref{Har=01} for all $0< \sigma < n/2$ and $-\infty < \alpha < 2\sigma$.   
Moreover, our results also differ from those in  \cite{CDQ} in the following two respects: 
\begin{itemize}
\item [$1$.] We consider the higher order fractional equation \eqref{Har=01} with the Hardy type nonlinearity $|x|^{-\alpha}u^p$ ($0< \alpha < 2\sigma$), which is not continuous at the origin. 

 \item [$2$.] Our results apply not only to the nonnegative entire solutions of \eqref{Har=01} in $\mathbb{R}^n$, but also to its singular solutions (which have a non-removable singularity at the origin).    
 
\end{itemize}

For the sake of application, we understand a solution $u$ of \eqref{Har=01} in the sense of distributions.   For $s\in \mathbb{R}$, we define  
$$
\mathcal{L}_s(\mathbb{R}^n):= \left\{ u \in L_{loc}^1(\mathbb{R}^n):  \int_{\mathbb{R}^n} \frac{|u(x)|}{(1 + |x|)^{n+2s}} dx < \infty  \right\}.  
$$
We say that $u$ is a nonnegative distributional solution of \eqref{Har=01} in $\mathbb{R}^n\backslash \{0\}$ if $u\in \mathcal{L}_\sigma(\mathbb{R}^n)$,  $|\cdot|^{-\alpha}u^{p} \in L_{loc}^1(\mathbb{R}^n \backslash \{0\})$, $u \geq 0$, and it satisfies  \eqref{Har=01} in the following sense:
\begin{equation}\label{Har=02}
\int_{\mathbb{R}^n} u(x) (-\Delta)^\sigma \varphi(x) dx = \int_{\mathbb{R}^n} |x|^{-\alpha}u(x)^{p} \varphi(x) dx ~~~~~ \textmd{for} ~ \textmd{any} ~ \varphi \in C_c^\infty(\mathbb{R}^n \backslash \{0\}). 
\end{equation}
Similarly, $u$ is a nonnegative distributional solution of \eqref{Har=01} in $\mathbb{R}^n$  if $u\in \mathcal{L}_\sigma(\mathbb{R}^n)$,  $|\cdot|^{-\alpha}u^{p} \in L_{loc}^1(\mathbb{R}^n)$, $u \geq 0$, and it satisfies  \eqref{Har=01} in the entire $\mathbb{R}^n$ in the following sense: 
\begin{equation}\label{Har=02h7p}
\int_{\mathbb{R}^n} u(x) (-\Delta)^\sigma \varphi(x) dx = \int_{\mathbb{R}^n} |x|^{-\alpha}u(x)^{p} \varphi(x) dx ~~~~~ \textmd{for} ~ \textmd{any} ~ \varphi \in C_c^\infty(\mathbb{R}^n).  
\end{equation}
Then we have the following integral characterization for nonnegative solutions to the equation \eqref{Har=01} when $0< \sigma  < n/2$ and $-\infty <  \alpha < 2\sigma$. 

\begin{theorem}\label{THM=02}
Let $0< \sigma  < n/2$ and $-\infty <  \alpha < 2\sigma$.  Suppose that one of the following holds:
\begin{itemize}
\item [$(1)$] $u \in \mathcal{L}_\sigma(\mathbb{R}^n)\cap C(\mathbb{R}^n)$ is a nonnegative distributional solution of \eqref{Har=01} in $\mathbb{R}^n$ with  $p >1$.  

\item [$(2)$] $u \in \mathcal{L}_\sigma(\mathbb{R}^n)\cap C(\mathbb{R}^n \backslash \{0\})$ is a nonnegative distributional solution of \eqref{Har=01} in $\mathbb{R}^n \backslash \{0\}$ with  $p \geq  \frac{n-\alpha}{n-2\sigma}$. 
\end{itemize}
Then $u$ satisfies the integral equation 
\begin{equation}\label{Int=21001}
u(x)= C_{n,\sigma} \int_{\mathbb{R}^n} \frac{u(y)^{p}}{|y|^{\alpha} |x-y|^{n-2\sigma}} dy~~~~~ \textmd{for} ~ x \in \mathbb{R}^n \backslash \{0\} 
\end{equation} 
with some positive constant $C_{n,\sigma}$. 
\end{theorem} 

\begin{remark}\label{Re=098m01}
We give several remarks on Theorem \ref{THM=02}.  
\begin{itemize}
\item [$(1)$] When $\sigma \in (0, n/2)$ is an integer,  instead of the assumption $u\in \mathcal{L}_\sigma(\mathbb{R}^n)$ in Theorem \ref{THM=02},  we only need $u \in L_{loc}^1(\mathbb{R}^n)$ since we have a better bound in this setting by the rescaling method  (see Proposition \ref{mInt=0hg01} in Section \ref{S2}).  

\item [$(2)$] If the assumption $(1)$ of Theorem \ref{THM=02} is satisfied, then $u$ satisfies the integral equation  \eqref{Int=21001} for all $x \in \mathbb{R}^n$. 

\item [$(3)$] It is easy to see that if $u \in C(\mathbb{R}^n)$ is a nonnegative solution of the integral equation \eqref{Int=21001}, then $u$ also satisfies the differential equation \eqref{Har=01} in $\mathbb{R}^n$.   
\end{itemize}
\end{remark}

Note that  Theorem \ref{THM=02} holds for both regular solutions and singular solutions near the origin, and only requires the growth restriction $u \in \mathcal{L}_\sigma(\mathbb{R}^n)$ near infinity which is slightly weaker than the growth restriction $u \in \mathcal{L}_{\frac{\nu}{2}}(\mathbb{R}^n)$ for some $0< \nu < 2$ satisfying $\sigma=m + \frac{\nu}{2}$  used in  \cite{CDQ} when $\sigma >1$ is not an integer.   Moreover, as pointed out in Remark \ref{Re=098m01} $(1)$ above,  such a growth restriction is not required in the case where $\sigma$ is an integer.

For $\alpha=0$ and $p=\frac{n+2\sigma}{n-2\sigma}$,  as a direct result of Theorem \ref{THM=02} and the classification results for positive $L_{loc}^{2n/(n-2\sigma)}(\mathbb{R}^n)$  solutions  of the integral equation \eqref{Int=21001} in Chen-Li-Ou \cite{CLO06} or  Li \cite{Li04},  we give a  classification for nonnegative solutions to the following conformally invariant equation with higher order fractional Laplacians
\begin{equation}\label{CONIn=001}
(-\Delta)^\sigma = u^{\frac{n+2\sigma}{n-2\sigma}} ~~~~~~  \textmd{in} ~ \mathbb{R}^n. 
\end{equation} 
\begin{theorem}\label{THM=02hy01}
Let $0< \sigma  < n/2$.  Suppose that  $u \in \mathcal{L}_\sigma(\mathbb{R}^n)\cap C(\mathbb{R}^n)$ is a nonnegative distributional solution of \eqref{CONIn=001} in $\mathbb{R}^n$.  Then either $u \equiv 0$ or $u$ must be of the form  
\begin{equation}\label{CONIn=002}
u(x)= \gamma_{n,\sigma} \left( \frac{\mu}{\mu^2 + |x -x_0|^2} \right)^{\frac{n-2\sigma}{2}}
\end{equation} 
for a fixed constant $\gamma_{n,\sigma}>0$ and for some $\mu >0$ and $x_0 \in \mathbb{R}^n$.   
\end{theorem} 

It is well-known that such classification results have many important applications in conformal geometry (see, e.g., \cite{JLX,JLX17} and the references therein). When $\sigma=m \in (0, n/2)$ is an integer, Theorem \ref{THM=02hy01} has been proved by Caffarelli-Gidas-Spruck \cite{CGS} for $m =1$, Lin \cite{Lin98} for $m=2$ and Wei-Xu \cite{WX} for $m \geq 3$.  When $\sigma \in (0, 1)$, Theorem \ref{THM=02hy01} was  proved by Chen-Li-Li \cite{CLL17} and Jin-Li-Xiong \cite{JLX}.  For general $\sigma \in (0, n/2)$, Chen, Li and Ou in \cite{CLO06} showed Theorem \ref{THM=02hy01} for nonnegative weak solutions $u \in H^\sigma (\mathbb{R}^n)$ of \eqref{CONIn=001} by using an integral representation.    For non-integral $\sigma \in (1, n/2)$, Cao, Dai and Qin in \cite{CDQ} recently proved Theorem \ref{THM=02hy01} for nonnegative classical solutions $u \in C_{loc}^{2\sigma + \varepsilon}(\mathbb{R}^n) \cap \mathcal{L}_{\frac{\nu}{2}}(\mathbb{R}^n)$ of \eqref{CONIn=001} with $\sigma=m + \frac{\nu}{2}$ for some $0< \nu < 2$, where $\varepsilon>0$ is arbitrarily small.  Clearly, such weak solutions and classical solutions of \eqref{CONIn=001} both satisfy the definition of distributional solutions in this paper. Thus,  our  Theorem \ref{THM=02hy01} in a slightly weaker assumption gives the classification of nonnegative solutions to \eqref{CONIn=001} in higher order fractional cases.

As an application of Theorem \ref{THM=02}, we prove a Liouville-type theorem for nonnegative solutions of \eqref{Har=01} in $\mathbb{R}^n$ for all $\sigma  \in (0, n/2)$ and $\alpha \in (-\infty , 2\sigma)$ in the following subcritical range 
\be\label{Sub=range=001}
1 < p < p_{\sigma,\alpha}^*:=\frac{n+2\sigma -2\alpha}{n-2\sigma}, 
\ee
where $p_{\sigma,\alpha}^*$ is the critical exponent.  
\begin{theorem}\label{THM=02App02}
Let $0< \sigma  < n/2$ and $-\infty <  \alpha < 2\sigma$.  Suppose that $u \in \mathcal{L}_\sigma(\mathbb{R}^n)\cap C(\mathbb{R}^n)$ is a nonnegative distributional solution of \eqref{Har=01} in $\mathbb{R}^n$ with  $1 < p < p_{\sigma,\alpha}^*$.  Then $u \equiv 0$ in $\mathbb{R}^n$. 
\end{theorem}

Liouville-type theorems for elliptic equations have been extensively studied,  which, in particular,  are crucial in deriving a priori estimates and establishing existence of positive solutions to non-variational boundary value problems. When $\alpha=0$, Liouville-type result in Theorem \ref{THM=02App02} has been proved in \cite{CGS,CLL17,GS,JLX,Li04,Lin98,WX,ZCCY} for different values of $\sigma \in (0, n/2)$.   For the Hardy-H\'enon's case $\alpha \ne 0$,  Theorem \ref{THM=02App02} was showed in \cite{DQ,GW,RZ,WLH} for $\sigma=1$, in \cite{DQ,BQ} for $0< \sigma <1$  and in \cite{DQ,NY} when $\sigma = m \in (1, n/2)$ is an integer. One may also see \cite{ChengLiu,DPQ,DZ} and the references therein for some earlier partial results. For the case where $\sigma \in (1, n/2)$ is not an integer, Theorem \ref{THM=02App02} with $\alpha \leq 0$ was recently proved by Cao, Dai and Qin in \cite{CDQ} for nonnegative classical solutions $u \in C_{loc}^{2\sigma + \varepsilon}(\mathbb{R}^n) \cap \mathcal{L}_{\frac{\nu}{2}}(\mathbb{R}^n)$ with $\sigma=m + \frac{\nu}{2}$ for some $0< \nu < 2$.  Hence,  Theorem \ref{THM=02App02} not only covers the existing gap when $0 < \alpha < 2\sigma$  and $\sigma \in (1, n/2)$ is not an integer,  but also gives a unified approach for both $\alpha<0$ and $\alpha \geq 0$ in higher order fractional cases.  

The Liouville-type result in Theorem \ref{THM=02App02} is optimal in the sense that for any $p > p_{\sigma,\alpha}^*$, the equation \eqref{Har=01} always admits positive solutions in $\mathbb{R}^n$. More specifically, we have the following existence result in the supercritical case.   

\begin{theorem}\label{THMExi01}
Let $0< \sigma  < n/2$ and $-\infty <  \alpha < 2\sigma$. For any $p > p_{\sigma,\alpha}^*$, the equation \eqref{Har=01} has a bounded positive solution $u \in \mathcal{L}_\sigma(\mathbb{R}^n)\cap C(\mathbb{R}^n)$.   
\end{theorem} 

When $\sigma \in (0, n/2)$ is an integer, the existence of positive solutions to \eqref{Har=01} in $\mathbb{R}^n$ for the critical and supercritical cases has been established via various methods; see, e.g., \cite{GG,GS,LV,Lions,LGZ,N1,NY} and the references therein. For the general $ \sigma \in(0, n/2)$ and $0 < \alpha < 2\sigma$, the existence in the critical case $p=p_{\sigma,\alpha}^*$ can be obtained by using the higher order fractional Hardy-Sobolev inequality (see, e.g., Yang \cite{YJ}). For $\sigma \in (0, 1)$ and $\alpha < 0$, the existence in the critical case has recently been showed by Barrios-Quaas \cite{BQ}. In combination with the integral equation \eqref{Int=21001}, their method should also work in the higher order case $\sigma \in (1, n/2)$. Hence, we are mainly concerned with the supercritical case. Theorem \ref{THMExi01} is unified for all $\sigma \in (0, n/2)$ and $\alpha \in (-\infty,  2\sigma)$ to establish the existence of \eqref{Har=01} in the supercritical case $p > p_{\sigma,\alpha}^*$.

As a consequence of Theorem \ref{THMExi01}, we have the existence of fast-decay singular solutions to the higher order fractional Lane-Emden equation
\be\label{24=La-Em} 
(-\Delta)^\sigma u = u^p ~~~~~~  \textmd{in} ~ \mathbb{R}^n \backslash \{0\}. 
\ee
These singular solutions could be used as building blocks to construct solutions of the singular Yamabe-type problems (see \cite{AW1,MP,HS} for $0 < \sigma \leq 1$ and $\sigma=2$).    

\begin{corollary}\label{24=La-Em-002}
Let $0< \sigma  < n/2$ and $\frac{n}{n-2\sigma} < p < \frac{n + 2\sigma}{n - 2\sigma}$. Then, for any $\varepsilon>0$ there exists a fast-decay positive singular solution $u_\varepsilon$ of \eqref{24=La-Em} such that
$$
u_\varepsilon(x) \sim 
\begin{cases}
A_{n,p,\sigma} |x|^{-\frac{2\sigma}{p-1}} ~~~~~ & \textmd{as} ~ x \to 0, \\
\varepsilon |x|^{-(n-2\sigma)} ~~~~~ & \textmd{as} ~ x \to \infty,
\end{cases}
$$
where $A_{n,p,\sigma}$ is a positive constant. 
\end{corollary}  

Applying Theorem \ref{THM=02}, we would also study radial symmetry of nonnegative solutions to \eqref{Har=01} in $\mathbb{R}^n$ with critical exponent or in $\mathbb{R}^n \backslash \{0\}$.  Symmetry is a fundamental property in elliptic PDEs which is especially important in searching solutions, classifying solutions, characterizing asymptotic behavior of solutions, and so on. In particular,  radial symmetry of solutions to an elliptic equation on $\mathbb{R}^n$ plays a crucial role in the classification of entire solutions; see, for example, \cite{CGS,CY,CL,GNN,Lin98,WX} and the references therein.  On the other hand,  the importance of studying singular solutions of Yamabe-type equations has been highlighted in the classical works of Schoen and Yau \cite{Sch88,SY88} on conformally flat manifolds. In general, the radial symmetry of singular solutions of a Yamabe-type equation on $\mathbb{R}^n \backslash \{0\}$ is an important step to understand asymptotic behavior for the corresponding equation near its isolated singularity \cite{CGS,HLT,KMPS}.  The recent works of Yang and Zou  \cite{YZa1,YZa2,YZa3} also showed that radial symmetry of global singular solutions has an important application in characterizing precise asymptotic behavior of solutions to fractional equations near isolated singularities, which, in particular, helps to overcome the difficulties caused by the lack of ODEs analysis for fractional equations.

We say that the origin $\{0\}$ is a non-removable singularity of the solution $u$ of \eqref{Har=01}  if $u$ cannot be extended to a continuous function near the origin.  Based on Theorem \ref{THM=02}, we show the following radial symmetry of nonnegative solutions to \eqref{Har=01} with Hardy-type weights ($0\leq \alpha <2\sigma$).  
\begin{theorem}\label{THM=01}
Let $0< \sigma  < n/2$ and $0\leq \alpha < 2\sigma$.  Suppose that  $u \in \mathcal{L}_\sigma(\mathbb{R}^n)\cap C(\mathbb{R}^n \backslash \{0\})$ is  a nonnegative distributional solution of \eqref{Har=01} in $\mathbb{R}^n \backslash \{0\}$, and $u$ has a  non-removable singularity at the origin when $\alpha=0$ and $p=\frac{n+2\sigma}{n-2\sigma}$.   
\begin{itemize}
\item [$(1)$] If $\frac{n-\alpha}{n-2\sigma} \leq    p \leq \frac{n+2\sigma-2\alpha}{n-2\sigma}$,  then $u$ is radially symmetric and monotonically decreasing with respect to the origin.  

\item [$(2)$] If $\frac{n+2\sigma-2\alpha}{n-2\sigma} < p \leq \frac{n+2\sigma-\alpha}{n-2\sigma}$, then $u$ is radially symmetric with respect to the origin.  
\end{itemize} 
\end{theorem} 

Clearly,  Theorem \ref{THM=01}  allows a solution $u$ to be singular at the origin.   In the case  $\alpha=0$ and $p=\frac{n+2\sigma}{n-2\sigma}$,  all of the solutions of \eqref{Har=01} with removable singularity have been classified in Theorem \ref{THM=02hy01}.  Note also that, for $0< \alpha < 2\sigma$, the exponent $\frac{n+2\sigma-\alpha}{n-2\sigma}$ in Theorem \ref{THM=01}  is bigger than the critical exponent $p_{\sigma,\alpha}^*=\frac{n+2\sigma -2\alpha}{n-2\sigma}$.    As a special case,  by Theorem \ref{THM=01} and Remark \ref{Re=098m01} $(1)$ before,   we have the following symmetry result for the higher order Hardy-H\'enon equations.

\begin{corollary}\label{Cor=0hj1}
Let  $0< m  < n/2$ be an integer and $0\leq \alpha < 2m$. Suppose that $u \in C(\mathbb{R}^n \backslash \{0\})$ is a nonnegative distributional solution of 
\be\label{Integqek5xfb} 
(-\Delta)^m u = |x|^{-\alpha}u^{p} ~~~~~~  \textmd{in} ~ \mathbb{R}^n \backslash \{0\}, 
\ee 
and $u$ has a  non-removable singularity at the origin when $\alpha=0$ and $p=\frac{n+2m}{n-2m}$.  If $\frac{n-\alpha}{n-2m} \leq  p \leq \frac{n+2m-\alpha}{n-2m}$, then $u$ is radially symmetric with respect to the origin.    
\end{corollary}

As mentioned earlier, these radial symmetry results would be useful in studying the higher order singular Yamabe problems and characterizing isolated singularities of the higher order Hardy-H\'enon equations.  In the case $\alpha=0$, the radial symmetry in Theorem \ref{THM=01} has been showed by Caffarelli-Gidas-Spruck \cite{CGS} for  $\sigma=1$,  by Lin \cite{Lin98} for $\sigma=2$ and by Caffarelli-Jin-Sire-Xiong \cite{CJSX} for $0< \sigma<1$.
In the case $0< \alpha <2\sigma$, Theorem \ref{THM=01} was proved by Li-Bao \cite{LB} for positive classical solutions of \eqref{Har=01} with non-removable singularity at the origin when $0< \sigma <1$ and $\frac{n-\alpha}{n-2\sigma} \leq  p \leq  \frac{n+2\sigma-2\alpha}{n-2\sigma}$.   Note that \cite{CJSX,LB} applied the well-known extension formula for the fractional Laplacian introduced by Caffarelli-Silvestre \cite{CS} to prove these symmetry results. Moreover, when $0< \alpha <2\sigma$, Theorem \ref{THM=01}  was established by Lu-Zhu \cite{LZ} for positive weak solutions $u \in H^\sigma (\mathbb{R}^n)$ of \eqref{Har=01} in $\mathbb{R}^n$ with the critical exponent $p=p_{\sigma,\alpha}^*=\frac{n+2\sigma-2\alpha}{n-2\sigma}$.  Hence, on the one hand,  instead of Caffarelli-Silvestre's extension, we use a unified approach to show radial symmetry of singular positive solutions of \eqref{Har=01} for all $\sigma \in (0, n/2)$ and $\alpha\in [0, 2\sigma)$.   On the other hand,  our  Theorem \ref{THM=01} in a weaker condition establishes the radial symmetry of entire solutions to \eqref{Har=01} in  $\mathbb{R}^n$  when $0< \alpha < 2\sigma$ and $p_{\sigma,\alpha}^* \leq p \leq \frac{n+2\sigma-\alpha}{n-2\sigma}$.

In the case $\alpha=0$ and $p=\frac{n+2\sigma}{n-2\sigma}$,  it is well-known that $c|x|^{-\frac{n-2\sigma}{2}}$ is a singular solution of \eqref{Har=01} with a positive constant $c$ depending only on $n$ and $\sigma$, the other singular solutions, e.g.,  the Fowler solutions,  of \eqref{Har=01} were obtained by Jin-Xiong \cite{JX19}.  In the same case,  the existence of Fowler-type singular solutions of \eqref{Har=01}  for $0< \sigma < 1$ and $\sigma=2$  was shown earlier by DelaTorre-del Pino-Gonz\'{a}lez-Wei \cite{DPGW} and Guo-Huang-Wang-Wei \cite{GHWW}, respectively.  We also remark that the classification of singular positive solutions to \eqref{Har=01} with $\alpha=0$ and $p=\frac{n+2\sigma}{n-2\sigma}$ was established by Fowler \cite{Fow} for $\sigma=1$ and by Frank-K\"{o}nig \cite{Fr-K19} for $\sigma=2$. 

The rest of this paper is organized as follows. In Section \ref{S2}, we establish the integral representation for \eqref{Har=01} stated in Theorem \ref{THM=02}. In Section \ref{S2=0hy}, we prove the Liouville-type theorem in Theorem \ref{THM=02App02}. In Section \ref{Ex}, we establish the existence results in Theorem \ref{THMExi01} and Corollary \ref{24=La-Em-002}. In Section \ref{S3}, we show the radial symmetry in Theorem \ref{THM=01}.  

In the following, we will use $B_r(x)$ to denote the open ball of radius $r$ in $\mathbb{R}^n$ with center $x$, and write $B_r(0)$ as $B_r$ for short. 

\vskip0.10in    
\vskip0.10in 

\noindent{\bf Acknowledgements.} The author is partially supported by NSFC grant 12301140. He would like to thank Professor Tianling Jin for his support and encouragement. He also thanks Professors Dong Ye and Qu\^{o}c Anh Ng\^{o} for helpful discussion on their work \cite{NY}. 

\section{An integral characterization}\label{S2}  

In this section, we prove Theorem \ref{THM=02} which gives an integral characterization of distributional solutions to \eqref{Har=01} for all $\sigma \in (0, n/2)$. Firstly, we show, under suitable assumptions, that a distributional solution $u$ of \eqref{Har=01} is always in $L_{loc}^p(\mathbb{R}^n, |x|^{-\alpha}dx)$ and satisfies the equation \eqref{Har=01} on the whole $\mathbb{R}^n$ in the sense of distributions. Let $\eta_1, \eta_2 \in C^\infty(\mathbb{R}^n)$ be two cut-off functions satisfying $0\leq \eta_1 \leq 1$ and 
$$
\eta_1(x)= 
\begin{cases}
0 ~~~~~~  \textmd{for} ~ |x| \leq 1, \\
1~~~~~~  \textmd{for} ~  |x| \geq 2, 
\end{cases}
~~~~ \eta_2(x) := 1- \eta_1(x).
$$
We define 
\begin{equation}\label{cut=oysbyl}
\eta_\varepsilon(x):=\eta_{1,\varepsilon}(x) \eta_2(x), ~~~~~~ \textmd{where} ~~  \eta_{1,\varepsilon}(x) :=\eta_1 \left( \frac{x}{\varepsilon} \right)~ \textmd{for} ~ \varepsilon>0. 
\end{equation}
Then one has the following estimate for $(-\Delta)^\sigma \eta_\varepsilon$, whose proof can be found in \cite[Lemma 2.1]{AGHW}. 
\begin{lemma}\label{Est-cut6nb5hs}
Let $\eta_\varepsilon$ be as in \eqref{cut=oysbyl}. Then for every $\sigma>0$ we have
\begin{equation}\label{cut=off02}
|(-\Delta)^\sigma\eta_\varepsilon(x)| \leq \frac{C_\sigma}{\varepsilon^{2\sigma}} \frac{1}{\big(1 + \frac{|x|}{\varepsilon}\big)^{n+2\sigma}} + \frac{C_\sigma}{(1 + |x|)^{n+2\sigma}}~~~~~ \textmd{for} ~ \textmd{all} ~  x\in \mathbb{R}^n.  
\end{equation}
\end{lemma}

Using Lemma \ref{Est-cut6nb5hs} and a dyadic decomposition argument as in Ao-Gonz\'{a}lez-Hyder-Wei \cite{AGHW},  we prove the following result which has also been obtained in \cite{AGHW} when $\alpha=0$.  See also \cite{CGS,JX19,Y} for the case with $\sigma \in (0, n/2)$ being an integer.      
\begin{proposition}\label{Glo-P01}
Suppose that  $0< \sigma  < n/2$ and $-\infty < \alpha < +\infty$.  Let $u \in \mathcal{L}_\sigma(\mathbb{R}^n)\cap C(\mathbb{R}^n \backslash \{0\})$ be a nonnegative distributional solution of \eqref{Har=01} in $\mathbb{R}^n \backslash \{0\}$ with $p>1$.   Suppose in addition that 
$$
p \geq  \frac{n-\alpha}{n-2\sigma}.  
$$
Then $|\cdot|^{-\alpha}u^p\in L^{1}_{\textmd{loc}} (\mathbb{R}^n)$ and $u$ is a distributional solution to the equation \eqref{Har=01} in $\mathbb{R}^n$, i.e., 
\begin{equation}\label{0725Har=02}
\int_{\mathbb{R}^n} u(x) (-\Delta)^\sigma \varphi(x) dx = \int_{\mathbb{R}^n} |x|^{-\alpha}u(x)^{p} \varphi(x) dx ~~~~~ \textmd{for} ~ \textmd{any} ~ \varphi \in C_c^\infty(\mathbb{R}^n). 
\end{equation}
\end{proposition} 

\begin{proof}
Let $\eta_\varepsilon$ be defined in \eqref{cut=oysbyl} with $\varepsilon \in (0, 1/2)$.   Taking $\eta_\varepsilon$ as a test function for the equation \eqref{Har=01} and using Lemma \ref{Est-cut6nb5hs} yields  
$$
\aligned
\int_{\mathbb{R}^n} |x|^{-\alpha} u^p \eta_\varepsilon dx &= \int_{\mathbb{R}^n} u (-\Delta)^\sigma\eta_\varepsilon dx  \\
& \leq  C + C \varepsilon^{-2\sigma} \int_{\mathbb{R}^n} \frac{u(x)}{\big(1 + \frac{|x|}{\varepsilon}\big)^{n+2\sigma}} dx \\
& \leq C + C \varepsilon^{-2\sigma} \left( \int_{B_{2\varepsilon}} + \int_{B_1 \backslash B_{2\varepsilon}} + \int_{B_1^c} \right) \frac{u(x)}{\big(1 + \frac{|x|}{\varepsilon}\big)^{n+2\sigma}} dx\\
& \leq C + C \varepsilon^{-2\sigma} \left( \int_{B_{2\varepsilon}} u dx +  \int_{B_1 \backslash B_{2\varepsilon}} \frac{u(x)}{\big(1 + \frac{|x|}{\varepsilon}\big)^{n+2\sigma}} dx + \varepsilon^{n+2\sigma}  \right).     
\endaligned 
$$
By H\"{o}lder inequality with $\frac{1}{p} + \frac{1}{p^\prime} =1$,  we have that for $p >1$,
$$
\aligned
\int_{B_1 \backslash B_{2\varepsilon}} \frac{u(x)}{\big(1 + \frac{|x|}{\varepsilon}\big)^{n+2\sigma}} dx  &\leq   \varepsilon^{n+2\sigma} \left( \int_{B_1 \backslash B_{2\varepsilon}} |x|^{-\alpha} u^p dx \right)^{1/p} \left( \int_{B_1 \backslash B_{2\varepsilon}} \frac{|x|^{\alpha p^{\prime}/p}}{(\varepsilon + |x|)^{p^\prime (n+2\sigma)}} dx \right)^{1/{p^\prime}} \\
& \leq  
\begin{cases}
C \varepsilon^{n + (\alpha -n )/p} \left(\int_{B_1 \backslash B_{2\varepsilon}} |x|^{-\alpha} u^p dx \right)^{1/p}  ~ & \textmd{if} ~ p > (\alpha -n) /2\sigma, \\
-C \varepsilon^{n + 2\sigma} \ln \varepsilon \left( \int_{B_1 \backslash B_{2\varepsilon}} |x|^{-\alpha} u^p dx \right)^{1/p} ~  & \textmd{if} ~ p \leq  (\alpha -n) /2\sigma. 
\end{cases}
\endaligned  
$$
Hence,  for $p >1$ and $p \geq  \frac{n-\alpha}{n-2\sigma}$ we have  
$$
\int_{\mathbb{R}^n} |x|^{-\alpha} u^p \eta_\varepsilon dx \leq C + C\varepsilon^{-2\sigma}  \int_{B_{2\varepsilon}} u dx + C \left(\int_{B_1 \backslash B_{2\varepsilon}} |x|^{-\alpha} u^p dx \right)^{1/p}, 
$$
and using Young inequality we obtain 
\begin{equation}\label{YHkhbdn=01}
\int_{\mathbb{R}^n} |x|^{-\alpha} u^p \eta_\varepsilon dx \leq C + C\varepsilon^{-2\sigma}  \int_{B_{2\varepsilon}} u dx.
\end{equation} 
Now, for every integer $i \geq -1$,  we get from H\"{o}lder inequality  and \eqref{YHkhbdn=01} that 
$$
\aligned
\int_{ \{ \frac{\varepsilon}{2^{i+1}} \leq |x| \leq \frac{\varepsilon}{2^{i}}  \} } u dx  & \leq  \left( \frac{\varepsilon}{2^{i}}  \right)^{n- (n-\alpha)/p} \left( \int_{ \{ \frac{\varepsilon}{2^{i+1}} \leq |x| \leq \frac{\varepsilon}{2^{i}}  \} }  |x|^{-\alpha} u^p dx  \right)^{1/p} \\
& \leq  C \left( \frac{\varepsilon}{2^{i}}  \right)^{n- (n + 2\sigma -\alpha)/p} \left( \int_{ \{ |x| \leq \frac{\varepsilon}{2^{i+1}} \} }  u dx  \right)^{1/p} + C \left( \frac{\varepsilon}{2^{i}}  \right)^{n- (n-\alpha)/p} \\
& \leq  C \left( \frac{\varepsilon}{2^{i}}  \right)^{n- (n + 2\sigma -\alpha)/p} \left( \int_{ \{ |x| \leq \varepsilon \} }  u dx  \right)^{1/p} + C \left( \frac{\varepsilon}{2^{i}}  \right)^{n- (n-\alpha)/p}. 
\endaligned 
$$
Note that we always have $n - (n + 2\sigma -\alpha)/p > 0$ since $n - (n + 2\sigma -\alpha)/p > 0 \iff p > (n+2\sigma-\alpha)/n$ which is true by 
$$
\begin{cases}
p >1 \geq  (n+2\sigma-\alpha)/n & ~~~~ \textmd{if} ~  (n+2\sigma-\alpha)/n \leq 1 \\
p \geq (n-\alpha) / (n-2\sigma) > (n+2\sigma-\alpha)/n  & ~~~~ \textmd{if} ~  (n+2\sigma-\alpha)/n > 1. 
\end{cases} 
$$
Thus, we sum the above inequality for all  integer $i \geq -1$ to  get 
$$
\int_{B_{2\varepsilon}} u dx \leq C \varepsilon^{n- (n + 2\sigma -\alpha)/p} \left( \int_{ \{ |x| \leq \varepsilon \} }  u dx  \right)^{1/p} + C \varepsilon^{n- (n-\alpha)/p}.  
$$
Using Young inequality again, we obtain
\be\label{sdjk834=0n}
\int_{B_{2\varepsilon}} u dx \leq C \varepsilon^{n- \frac{2\sigma -\alpha}{p-1}} + C \varepsilon^{n- \frac{n-\alpha}{p}}.  
\ee
Since $n- \frac{2\sigma -\alpha}{p-1} \geq 2\sigma$ and $n- \frac{n-\alpha}{p} \geq 2\sigma$ due to $p\geq \frac{n-\alpha}{n-2\sigma}$, combining  \eqref{sdjk834=0n} with \eqref{YHkhbdn=01}  and sending $\varepsilon \to 0$, we obtain that $\int_{B_1} |x|^{-\alpha} u^p dx < \infty$ and hence, $|\cdot|^{-\alpha}u^p\in L^{1}_{\textmd{loc}} (\mathbb{R}^n)$.    

To complete the proof, we next prove that $u$ is a distributional solution of \eqref{Har=01} in the whole space $\mathbb{R}^n$.  Since $|\cdot|^{-\alpha}u^p\in L^{1}_{\textmd{loc}} (\mathbb{R}^n)$, we only need to show that for every $\varphi \in C_c^\infty(\mathbb{R}^n)$, 
\be\label{Dis=0Asoq=001}
\int_{\mathbb{R}^n} u (-\Delta)^\sigma (\varphi \eta_{1,\varepsilon}) dx \to \int_{\mathbb{R}^n} u (-\Delta)^\sigma \varphi dx 
\ee
as $\varepsilon \to 0$, where $\eta_{1, \varepsilon}$ is defined as in \eqref{cut=oysbyl} with $ \varepsilon  \in (0, 1/4)$. For this purpose, we define
$$
I_\varepsilon(x):=(-\Delta)^\sigma (\varphi \eta_{1,\varepsilon}) (x) - \eta_{1,\varepsilon}(x) (-\Delta)^\sigma \varphi (x), ~~~~ x \in \mathbb{R}^n. 
$$
Let $k$ be the integer part of $\sigma$, that is, $\sigma := k + \sigma^\prime$ with $\sigma^\prime \in (0, 1)$ and $k \in \{0\} \cup \mathbb{N}$. We consider the two cases $|x| \geq 4\varepsilon$  and  $|x| \leq 4\varepsilon$ to estimate $I_\varepsilon(x)$.  

{\it Case $1$}:  $|x| \geq  4\varepsilon$. In this case, we note that all derivatives of $\eta_{1, \varepsilon}$ at $x$ are equal to $0$.  Then we have  (for simplicity we omit the constant $c_{n,\sigma} $ in \eqref{Fda-Lapcc=01})  
$$
\aligned
(-\Delta)^\sigma (\varphi \eta_{1,\varepsilon}) (x) & = \int_{\mathbb{R}^n}  \frac{\eta_{1,\varepsilon}(x) [(-\Delta)^k \varphi(x) - (-\Delta)^k \varphi (z)]  + (-\Delta)^k [\varphi(z) (1 - \eta_{1,\varepsilon}(z) )] }{|x - z|^{n +2\sigma^\prime}}  dz \\
& = \eta_{1,\varepsilon}(x) (-\Delta)^\sigma \varphi (x) + \int_{\mathbb{R}^n} \frac{ (-\Delta)^k [\varphi(z) (1 - \eta_{1,\varepsilon}(z) )] }{|x - z|^{n +2\sigma^\prime}}  dz \\
& = \eta_{1,\varepsilon}(x) (-\Delta)^\sigma \varphi (x) + C \int_{\mathbb{R}^n} \frac{  \varphi(z) (1 - \eta_{1,\varepsilon}(z) ) }{|x - z|^{n +2\sigma}}  dz  \\
\endaligned
$$
where the integration by parts is used in the last equality. Notice also that the integrand is not singular at $z=x$ since the function $1 - \eta_{1,\varepsilon}$ is supported  in $B_{2\varepsilon}$. Hence  we have 
\be\label{HK=Appo=081}
|I_\varepsilon(x)| \leq \int _{B_{2\varepsilon}} \frac{  |\varphi(z) (1 - \eta_{1,\varepsilon}(z) ) |}{|x - z|^{n +2\sigma}}  dz \leq  C \frac{\varepsilon^n}{|x|^{n+2\sigma}} ~~~~ \textmd{for} ~  |x| \geq 4\varepsilon. 
\ee

{\it Case $2$}:  $|x| \leq  4\varepsilon$.  In this case,  instead of estimating $I_\varepsilon(x)$ we estimate $(-\Delta)^\sigma (\varphi\eta_{1,\varepsilon})$. Denote $\varphi_\varepsilon :=\varphi\eta_{1,\varepsilon}$. In fact, we have  
$$
(-\Delta)^\sigma \varphi_\varepsilon (x) = \frac{1}{2} \int_{\mathbb{R}^n} \frac{2 (-\Delta)^k \varphi_\varepsilon(x) - (-\Delta)^k \varphi_\varepsilon (x+y) - (-\Delta)^k \varphi_\varepsilon (x-y)}{|y|^{n+2 \sigma^\prime}} dy. 
$$
Moreover,  there exists $C>0$ depending on $\varphi$ such that the integrand in the above integral  can be controlled by 
$$
\frac{C}{\varepsilon^{2k} |y|^{n+2\sigma^\prime }} \min \left\{ \frac{|y|^2}{\varepsilon^2},  1 \right\}. 
$$
Thus, we obtain 
$$
|(-\Delta)^\sigma \varphi_\varepsilon (x)| \leq \frac{C}{\varepsilon^{2\sigma}} ~~~~ \textmd{for} ~ |x| \leq 4\varepsilon. 
$$
Hence
\be\label{HK=Appo=082}
|I_\varepsilon(x)| \leq  \frac{C}{\varepsilon^{2\sigma}}  ~~~~ \textmd{for} ~ |x| \leq 4\varepsilon
\ee
due to the boundedness of $\eta_{1, \varepsilon} (-\Delta)^\sigma \varphi$.  Combining \eqref{HK=Appo=081} and \eqref{HK=Appo=082}, we get
\be\label{HK=Appo=083}
|I_\varepsilon(x)| \leq C \frac{1}{\varepsilon^{2\sigma}} \frac{\chi_{B_1}(x)}{\big( 1 + \frac{|x|}{\varepsilon} \big)^{n+2\sigma}} + C \varepsilon^n \frac{\chi_{B_1^c}(x)}{|x|^{n+2\sigma}} ~~~~\textmd{for} ~ \textmd{all}~ x \in \mathbb{R}^n.  
\ee
This is obviously true if $\sigma$ is an integer.  For $p^\prime >1$ with $\frac{1}{p} +\frac{1}{p^\prime} =1$,  we have 
$$
\aligned
\left\| |\cdot|^{\frac{\alpha}{p}}  I_\varepsilon \right\|_{L^{p^\prime}(B_1)}^{p^\prime} \leq C \varepsilon^{-2\sigma p^\prime} \int_{B_\varepsilon} |x|^{\frac{\alpha p^\prime}{p}} dx + C \varepsilon^{n p^\prime} \int_{\{\varepsilon \leq |x| \leq 1\}} |x|^{\frac{\alpha p^\prime}{p} - (n+2\sigma)p^\prime}\leq  C
\endaligned 
$$
with some $C >0$ independent of $\varepsilon$. Hence, for each $\delta \in (0, 1)$, 
$$
\int_{B_\delta} u(x) |I_\varepsilon(x)| dx \leq  \left\| |\cdot|^{-\frac{\alpha}{p}}  u \right\|_{L^{p}(B_\delta)} \left\| |\cdot|^{\frac{\alpha}{p}}  I_\varepsilon \right\|_{L^{p^\prime}(B_\delta)} \leq C \left\| |\cdot|^{-\frac{\alpha}{p}}  u \right\|_{L^{p}(B_\delta)}
$$
uniformly in $\varepsilon$ and $\delta$.  On the other hand, by the assumption $u \in \mathcal{L}_\sigma(\mathbb{R}^n)$ and \eqref{HK=Appo=083},   we have that for any fixed $\delta \in (0, 1)$, 
$$
\lim_{\varepsilon \to 0} \int_{B_\delta^c} u(x) |I_\varepsilon(x)| dx = 0. 
$$
Therefore,  by sending $\delta \to 0$ we obtain 
$$
\lim_{\varepsilon \to 0} \int_{\mathbb{R}^n} u(x) |I_\varepsilon(x)| dx = 0. 
$$
This implies that \eqref{Dis=0Asoq=001} holds and thus, $u$ is a distributional solution to the equation \eqref{Har=01} in $\mathbb{R}^n$. 
\end{proof}   

Now we give some growth estimates for solutions of \eqref{Har=01} near infinity. We start by recalling a result from \cite[Lemma 5.2]{AGHW}. 

\begin{lemma}\label{Aop-decay=01} 
Let $\psi \in C^\infty(\mathbb{R}^n)$ be such that $\psi(x)=\frac{1}{|x|^\tau}$ on $B_1^c$ for some $\tau>0$. Let $\xi \in C^\infty(\mathbb{R}^n)$ be a cut-off function satisfying 
$$
\xi(x) =1 ~~~ \textmd{for} ~ |x|\leq 1 ~~~~~~ \textmd{and} ~~~~~~ \xi(x) =0 ~~~ \textmd{for} ~ |x| \geq 2.  
$$
For any $\varepsilon>0$, we define $\xi_\varepsilon:=\xi(\varepsilon x)$ and $\psi_\varepsilon(x):= \psi  \xi_\varepsilon(x)$. Then for every $\sigma>0$ we have
$$
(-\Delta)^\sigma \psi_\varepsilon \to (-\Delta)^\sigma \psi ~~~~ \textmd{locally} ~ \textmd{uniformly} ~ \textmd{in} ~ \mathbb{R}^n  ~ \textmd{as} ~ \varepsilon \to 0.
$$
Moreover, there exists $C=C(n, \sigma, \varphi) >0$ (independent of $\varepsilon$) such that 
$$
|(-\Delta)^\sigma \psi_\varepsilon (x)| \leq C
\begin{cases}
(1 + |x|)^{-2\sigma-\tau} ~~~~~~ &\textmd{if} ~ \tau <n,\\
(1 + |x|)^{-2\sigma-\tau}\log(2 + |x|) ~~~~~~ &\textmd{if} ~ \tau =n,\\
(1 + |x|)^{-2\sigma-n} ~~~~~~ &\textmd{if} ~ \tau >n. 
\end{cases}
$$
\end{lemma}

Based on Lemma \ref{Aop-decay=01}, a bootstrap argument allows us to give the following growth estimate of solutions to \eqref{Har=01} near infinity. 

\begin{proposition}\label{Plo-P02}
Suppose that  $0< \sigma  < n/2$ and $-\infty < \alpha < 2\sigma$.  Let $u \in \mathcal{L}_\sigma(\mathbb{R}^n)$ be a nonnegative distributional solution of
\begin{equation}\label{0727Har=03}
(-\Delta)^\sigma u = |x|^{-\alpha}u^{p} ~~~~~~  \textmd{in} ~ \mathbb{R}^n 
\end{equation}
with  $|\cdot|^{-\alpha}u^p \in L_{\textmd{loc}}^1(\mathbb{R}^n)$  for some  $ p>1$.  
\begin{itemize}
\item [$(1)$] If $1< p < \frac{n-\alpha}{n-2\sigma}$, then
\begin{equation}\label{Glo-E002gyd}
\int_{\mathbb{R}^n} \frac{|x|^{-\alpha}u(x)^p}{1 + |x|^{\gamma}} dx < +\infty ~~~~~~ \textmd{for} ~ \textmd{every}  ~ \gamma > 0. 
\end{equation}

\item [$(2)$] If $p \geq  \frac{n-\alpha}{n-2\sigma}$, then
\begin{equation}\label{Glo-E002}
\int_{\mathbb{R}^n} \frac{|x|^{-\alpha}u(x)^p}{1 + |x|^{\gamma}} dx < +\infty ~~~~~~ \textmd{for} ~ \textmd{every}  ~ \gamma > n - 2\sigma +\frac{\alpha-2\sigma}{p-1}.  
\end{equation}
\end{itemize}
\end{proposition}

\begin{remark}\label{Marencgej76} 
When $\alpha=0$, the growth estimate in \eqref{Glo-E002} has been proved by  Ao-Gonz\'{a}lez-Hyder-Wei \cite{AGHW} for $\gamma=n-2\sigma$.   Here we continue to use the bootstrap argument to obtain the sharp estimate \eqref{Glo-E002},  which is essential in showing the integral characterization of solutions. We also remark that such sharp growth estimate was obtained by Du-Yang \cite{DY} when $\alpha=0$, $p=\frac{n+2\sigma}{n-2\sigma}$ and  $\sigma$ is an integer by using a more direct argument (see also Proposition \ref{mInt=0hg01} below) that applies only to equations of integer order if $\sigma > 1$.      
\end{remark} 
 
\begin{proof}
For any $\tau>0$ we take $\psi\in C^\infty(\mathbb{R}^n)$ satisfying $\psi(x)=\frac{1}{|x|^{n+\tau}}$ on $B_1^c$. Let $\xi_\varepsilon$ be defined as in Lemma \ref{Aop-decay=01} and let $\psi_\varepsilon:=\psi\xi_\varepsilon$. Then, using Lemma \ref{Aop-decay=01}, dominated convergence theorem and monotone convergence theorem we obtain   
\begin{equation}\label{YopZ=001}
\int_{\mathbb{R}^n} \frac{|x|^{-\alpha} u(x)^p}{1 + |x|^{n+\tau}}  dx =\int_{\mathbb{R}^n} u(-\Delta)^\sigma \psi dx <\infty. 
\end{equation} 
For any $q > n +\frac{\alpha}{p}$ we write $q=q_1 + q_2$ with $q_1> \frac{n}{p}$ and $q_2 > \frac{n}{p^\prime} + \frac{\alpha}{p}$, where $p^\prime$ is the conjugate exponent of $p$.  By H\"{o}lder inequality and \eqref{YopZ=001} we have  
$$
\aligned
\int_{\mathbb{R}^n} \frac{u(x)}{1 + |x|^q} dx  &\leq C_1 + \int_{B_1^c} \frac{|x|^{-\alpha/p}u(x) }{|x|^{q_1}} \cdot \frac{|x|^{\alpha/p}}{|x|^{q_2}} dx \\
&\leq C_1 + \left( \int_{B_1^c} \frac{|x|^{-\alpha}u(x)^p }{|x|^{q_1p}} dx \right)^{1/p}\left( \int_{B_1^c} \frac{|x|^{\alpha p^\prime/p} }{|x|^{q_2 p^\prime}} dx \right)^{1/p^\prime} \\
& < \infty ~~~~~~ \textmd{for} ~ \textmd{every} ~ q > n+ \frac{\alpha}{p}.  
\endaligned
$$
Thus,  the integrability of $u$ lifts from $u \in \mathcal{L}_\sigma(\mathbb{R}^n)$ to $u \in \mathcal{L}_s(\mathbb{R}^n)$ for any $s > \frac{\alpha}{2p}$.  Next, we consider two cases for $p>1$ separately.   

{\it Case $1$:}  $p \leq \frac{-\alpha}{n-2\sigma}$. In this case, we have $n +\frac{\alpha}{p}\leq 2\sigma$.  By Lemma \ref{Aop-decay=01},  for any $\tau>0$ we can take $\psi \in C^\infty(\Rn)$ with $\psi(x)=\frac{1}{|x|^\tau}$ on $B_1^c$ as a test function to \eqref{0727Har=03}, and consequently we have 
$$
\int_{\Rn} \frac{|x|^{-\alpha} u(x)^p}{1 + |x|^\tau} dx < \infty~~~~~~ \textmd{for} ~ \textmd{every} ~ \tau > 0. 
$$

{\it Case $2$:} $p > \frac{-\alpha}{n-2\sigma}$. In this case we have $0< n + \frac{\alpha}{p} - 2\sigma <n$. By Lemma \ref{Aop-decay=01}, we use  $\psi \in C^\infty(\Rn)$ with $\psi(x)=\frac{1}{|x|^{n-2\sigma+\alpha/p +\tau}}$ on $B_1^c$ for small $\tau>0$ as a test function to \eqref{0727Har=03}, and thus we get 
$$
\int_{\Rn} \frac{|x|^{-\alpha} u(x)^p}{1 + |x|^q} dx \leq C \int_{\Rn} \frac{u(x)}{1 + |x|^{q+2\sigma}} dx< \infty~~~~~~ \textmd{for} ~ \textmd{every} ~q > n+\frac{\alpha}{p} - 2\sigma.   
$$
Again by H\"{o}lder inequality 
$$
\int_{\Rn} \frac{u(x)}{1 + |x|^q} dx < \infty~~~~~~ \textmd{for} ~ \textmd{every} ~q > n+\frac{\alpha}{p} + \frac{\alpha}{p^2} - \frac{2\sigma}{p}.  
$$
Obviously, this further lifts the integrability of $u$. As discussed before we are still divided into two cases. 

{\it Case $1^\prime$:} $p \leq \frac{-(\alpha-2\sigma) + \sqrt{(\alpha-2\sigma)^2 - 4\alpha(n-2\sigma)}}{2(n-2\sigma)}$. In this case we know $n+\frac{\alpha}{p} + \frac{\alpha}{p^2} - \frac{2\sigma}{p} \leq 2\sigma$. Using Lemma \ref{Aop-decay=01} and taking $\psi \in C^\infty(\Rn)$ with $\psi(x)=\frac{1}{|x|^\tau}$ on $B_1^c$ for $\tau>0$ as a test function yields 
$$
\int_{\Rn} \frac{|x|^{-\alpha} u(x)^p}{1 + |x|^\tau} dx < \infty~~~~~~ \textmd{for} ~ \textmd{every} ~ \tau > 0. 
$$

{\it Case $2^\prime$:} $p > \frac{-(\alpha-2\sigma) + \sqrt{(\alpha-2\sigma)^2 - 4\alpha(n-2\sigma)}}{2(n-2\sigma)}$. In this case we have $0< n+\frac{\alpha}{p} + \frac{\alpha}{p^2} - \frac{2\sigma}{p} - 2\sigma< n$.  Using Lemma \ref{Aop-decay=01} and taking $\psi \in C^\infty(\Rn)$ with $\psi(x)=\frac{1}{|x|^q}$ on $B_1^c$ for $q> n+\frac{\alpha}{p} + \frac{\alpha}{p^2} - \frac{2\sigma}{p} - 2\sigma$ as a test function yields 
$$
\int_{\Rn} \frac{|x|^{-\alpha} u(x)^p}{1 + |x|^q} dx  < \infty~~~~~~ \textmd{for} ~ \textmd{every} ~q > n+\frac{\alpha}{p} + \frac{\alpha}{p^2} - \frac{2\sigma}{p} - 2\sigma. 
$$
Again by H\"{o}lder inequality 
$$
\int_{\Rn} \frac{u(x)}{1 + |x|^q} dx < \infty~~~~~~ \textmd{for} ~ \textmd{every} ~q > n+ \left( \frac{\alpha}{p} + \frac{\alpha}{p^2} + \frac{\alpha}{p^3} \right) - \left( \frac{2\sigma}{p} + \frac{2\sigma}{p^2} \right).   
$$
This lifts the integrability of $u$ again.  

Continuing this process, after finitely many steps we will get the desired conclusion for any fix $p>1$ and fix $\gamma$ satisfying the assumption.    
The proof is completed. 
\end{proof}

When $\sigma \in (0, n/2)$ is an integer, the growth estimates of solutions to \eqref{0727Har=03} are much easier to obtain by the standard rescaling method.  More precisely, we have 

\begin{proposition}\label{mInt=0hg01}
Suppose that  $\sigma \in (0,  n/2)$ is an integer and $ \alpha \in (-\infty, 2\sigma)$.  Let $u\in L_{loc}^1(\Rn)$ be a nonnegative distributional solution of \eqref{0727Har=03} with  $|\cdot|^{-\alpha}u^p \in L_{\textmd{loc}}^1(\mathbb{R}^n)$  for some  $ p>1$.  Then
\be\label{Gammhy=01}
\int_{\Rn} \frac{u(x)}{(1 + |x|)^\gamma} dx < \infty ~~~~~~ \textmd{for} ~ \textmd{every} ~ \gamma > n+\frac{\alpha - 2\sigma}{p-1}  
\ee
and
\be\label{Gammhy=02}
\int_{\Rn} \frac{|x|^{-\alpha}u(x)^p}{(1 + |x|)^\gamma} dx < \infty ~~~~~~ \textmd{for} ~ \textmd{every} ~ \gamma > n -2\sigma +\frac{\alpha-2\sigma}{p-1}. 
\ee

\end{proposition} 

\begin{proof}
Let $\varphi \in C_c^\infty(\Rn)$ be a nonnegative function satisfying $\varphi=1$ on $B_1$ and $\varphi=0$ on $B_2^c$. For $R >0$, set $\varphi_R(x):=\varphi\left( \frac{x}{R} \right)$. Fix an integer $q > \frac{2\sigma p}{p-1}$ and choose $\varphi_R^q$ as a test function in \eqref{0727Har=03}.  Note that here $\sigma \in (0, n/2)$ is an integer, we have $|(-\Delta)^\sigma \varphi_R^q(x)| \leq C R^{-2\sigma} \varphi_R(x)^{q-2\sigma}$ for some constant $C>0$.    By H\"{o}lder inequality we obtain 
$$
\aligned
\int_{\Rn} |x|^{-\alpha}u^p \varphi_R^q dx  & =\int_{\Rn} u (-\Delta)^\sigma \varphi_R^q dx \leq \frac{C}{R^{2\sigma}} \int_{B_{2R} \backslash B_R} u \varphi_R^{q-2\sigma} dx \\
& = \frac{C}{R^{2\sigma}} \int_{B_{2R} \backslash B_R} |x|^{-\alpha/p}u\varphi_R^{q/p}  \cdot |x|^{\alpha/p} \varphi_R^{q-q/p -2\sigma} dx \\
&\leq \frac{C}{R^{2\sigma}} \left( \int_{B_{2R} \backslash B_R} |x|^{-\alpha}u^p \varphi_R^{q} dx \right)^{1/p}  \left( \int_{B_{2R} \backslash B_R} |x|^{\alpha p^\prime/p} dx \right)^{1/p^\prime}\\
& \leq C R^{\frac{\alpha}{p} + \frac{n}{p^\prime} -2\sigma} \left( \int_{B_{2R} \backslash B_R} |x|^{-\alpha}u^p \varphi_R^{q} dx \right)^{1/p}. 
\endaligned
$$
This gives that 
\be\label{hxnd7g72=06}
\int_{B_R} |x|^{-\alpha}u^p  dx \leq C R^{n +\frac{\alpha p^\prime}{p} -2\sigma p^\prime}=C R^{n -2\sigma +\frac{\alpha-2\sigma}{p-1}} ~~~~~~ \textmd{for} ~ \textmd{any} ~ R>0. 
\ee
{\it Here we also point out that the above two estimates imply $u \equiv 0$ if $1< p \leq \frac{n-\alpha}{n-2\sigma}$}. Using H\"{o}lder inequality again, we obtain, for any $\gamma \in \mathbb{R}$ and $R \geq 1$,  that 
$$
\aligned
\int_{B_{2R} \backslash B_R} \frac{u(x)}{(1 + |x|)^\gamma} dx & \leq \left(  \int_{B_{2R}\backslash B_R} |x|^{-\alpha}u^p dx \right)^{1/p} \left(  \int_{B_{2R}\backslash B_R} \frac{|x|^{\alpha p^\prime /p}}{(1 + |x|)^{\gamma p\prime}} dx \right)^{1/p^\prime} \\
& \leq C R^{n+\frac{\alpha-2\sigma}{p-1} -\gamma}. 
\endaligned
$$
Hence, when $\gamma> n+\frac{\alpha - 2\sigma}{p-1}$, a dyadic sum argument leads to that \eqref{Gammhy=01} holds. 

On the other hand, estimate \eqref{hxnd7g72=06} also implies 
$$
\aligned
\int_{\Rn}\frac{|x|^{-\alpha}u(x)^p}{(1 + |x|)^\gamma} dx & =\int_{B_1}\frac{|x|^{-\alpha}u(x)^p}{(1 + |x|)^\gamma} dx + \sum_{i=0}^\infty \int_{B_{2^{i+1}} \backslash B_{2^i}}  \frac{|x|^{-\alpha}u(x)^p}{(1 + |x|)^\gamma} dx \\
& \leq C_1 + C \sum_{i=0}^\infty (2^i)^{n -2\sigma +\frac{\alpha-2\sigma}{p-1} -\gamma} < \infty
\endaligned
$$
if $\gamma > n -2\sigma +\frac{\alpha-2\sigma}{p-1}$. This proves \eqref{Gammhy=02}.  
\end{proof} 

Suppose that  $0< \sigma  < n/2$ and $-\infty < \alpha < 2\sigma$.  Let $u \in \mathcal{L}_\sigma(\mathbb{R}^n)  \cap C(\mathbb{R}^n \backslash \{0\})$ be a nonnegative distributional solution of \eqref{Har=01} with $p \geq  \frac{n-\alpha}{n-2\sigma}$.  Then,  by Propositions \ref{Glo-P01} and \ref{Plo-P02},  the following function
\begin{equation}\label{FS098}
v(x):=C_{n,\sigma} \int_{\mathbb{R}^n} \frac{|y|^{-\alpha}u(y)^{p}}{|x - y|^{n-2\sigma}} dy
\end{equation}
is well-defined for every $x\in \mathbb{R}^n \backslash \{0\}$ and continuous on $\mathbb{R}^n \backslash \{0\}$, where  $C_{n,\sigma}$ is a positive constant such that $C_{n,\sigma} |\cdot|^{2\sigma-n}$ is the fundamental solution of the fractional Laplacian.     For any $R>0$, we can write $v$ as $v=v_{1,R} + v_{2,R}$ with
$$
v_{1,R} (x) = \int_{B_{2R}} \frac{|y|^{-\alpha}u(y)^{p}}{|x - y|^{n-2\sigma}} dy ~~~~ \textmd{and}   ~~~~
v_{2,R} (x) = \int_{B_{2R}^c} \frac{|y|^{-\alpha}u(y)^{p}}{|x - y|^{n-2\sigma}}  dy. 
$$
From  Proposition \ref{Glo-P01} we know that $|\cdot|^{-\alpha}u^{p} \in L^1(B_{2R})$ and hence $v_{1,R} \in L^1(B_R)$.   By Proposition \ref{Plo-P02} we easily get $v_{2,R} \in L^\infty(B_R)$.  Thus $v\in L_{loc}^1(\mathbb{R}^n)$.   Moreover, we also have the following growth estimate for $v$
 near infinity. 
\begin{proposition}\label{VVV}
Suppose that  $0< \sigma  < n/2$ and $-\infty < \alpha < 2\sigma$.  

\begin{itemize}

\item [$(1)$] Let $u \in \mathcal{L}_\sigma(\mathbb{R}^n)\cap C(\mathbb{R}^n \backslash \{0\})$ be a nonnegative distributional solution of \eqref{Har=01} with $ p \geq  \frac{n-\alpha}{n-2\sigma}$,  and  let $v$ be defined by \eqref{FS098}.  Then we have $v\in \mathcal{L}_s (\mathbb{R}^n)$ for any $s > \frac{\alpha-2\sigma}{2(p-1)}$.  

\item [$(2)$]  Let $u \in \mathcal{L}_\sigma(\mathbb{R}^n)\cap C(\mathbb{R}^n)$ be a nonnegative distributional solution of \eqref{Har=01} in $\mathbb{R}^n$ with $1 < p < \frac{n-\alpha}{n-2\sigma}$,   and  let $v$ be defined by \eqref{FS098}.  Then we have $v\in \mathcal{L}_s (\mathbb{R}^n)$ for any $s > \frac{2\sigma - n}{2}$.  

\end{itemize}

\end{proposition}
\begin{proof}
By  Fubini's theorem, for any $q > 2\sigma$, we have 
\begin{equation}\label{Fubi}
\int_{\mathbb{R}^n} \frac{v(x)}{(1+|x|)^{q}} dx =c_{n,\sigma} \int_{\mathbb{R}^n} |y|^{-\alpha} u(y)^{p} \left( \int_{\mathbb{R}^n} \frac{1}{|x-y|^{n-2\sigma}} \frac{1}{(1+|x|)^q} dx\right) dy. 
\end{equation} 
If $|y| \leq 1$, then
$$
\aligned
\int_{\mathbb{R}^n} \frac{1}{|x-y|^{n-2\sigma}} \frac{1}{(1+|x|)^q} dx & \leq C \left( \int_{B_3} \frac{1}{|x|^{n-2\sigma}}dx +  \int_{B_2^c} \frac{1}{|x|^{n-2\sigma+q}}  dx  \right) \\ 
& \leq C <\infty.
\endaligned 
$$
If $|y| >1$, then
$$
\int_{\mathbb{R}^n} \frac{1}{|x-y|^{n-2\sigma}} \frac{1}{(1+|x|)^q} dx  =: \sum_{i=1}^3 I_i,
$$
where
$$
I_1=\int_{ \{|x|\leq \frac{|y|}{2}\} }    \frac{1}{|x-y|^{n-2\sigma}} \frac{1}{(1+|x|)^q}  dx \leq  C 
\begin{cases} 
|y|^{2\sigma - q},  ~~~ & \textmd{if} ~ q < n, \\
\log (1 + |y|)|y|^{2\sigma-n}, ~~~& \textmd{if} ~ q = n,\\
|y|^{2\sigma - n}, ~~~& \textmd{if} ~ q > n, 
\end{cases}
$$
$$
\aligned
I_2 & = \int_{\{ \frac{|y|}{2} < |x| < 2|y| \} }  \frac{1}{|x-y|^{n-2\sigma}} \frac{1}{(1+|x|)^q} dx  \leq \frac{C}{|y|^{q}}\int_{\{ \frac{|y|}{2} < |x| < 2|y| \}}  \frac{1}{|x-y|^{n-2\sigma}} dx \\ 
& \leq C |y|^{2\sigma -q}
\endaligned
$$
and
$$
I_3=\int_{\{ |x|\geq  2|y| \}} \frac{1}{|x-y|^{n-2\sigma}} \frac{1}{(1+|x|)^q} dx \leq C  \int_{\{ |x|\geq  2|y| \}} \frac{1}{|x|^{n-2\sigma+q}} dx \leq  C |y|^{2\sigma-q}.  
$$

Let $u \in \mathcal{L}_\sigma(\mathbb{R}^n)\cap C(\mathbb{R}^n \backslash \{0\})$ be a nonnegative distributional solution of \eqref{Har=01} with $ p \geq  \frac{n-\alpha}{n-2\sigma}$.  Then  for $q > n + \frac{\alpha-2\sigma}{p-1}$ (note that $n + \frac{\alpha-2\sigma}{p-1} \geq 2\sigma$ due to $p\geq \frac{n - \alpha}{n-2\sigma}$),  by  \eqref{Fubi} and \eqref{Glo-E002}  we get 
$$
\int_{\mathbb{R}^n} \frac{v(x)}{(1+|x|)^{q}} dx \leq C \int_{B_1} |y|^{-\alpha}u(y)^{p} dy +C \int_{B_1^c}   \frac{\log (1 + |y|) |y|^{-\alpha}u(y)^{p}}{|y|^{q-2\sigma}} dy  <\infty. 
$$
That is,  $v \in \mathcal{L}_s(\mathbb{R}^n)$ for any $s > \frac{\alpha-2\sigma}{2(p-1)}$.   

Similarly,  if $u \in \mathcal{L}_\sigma(\mathbb{R}^n)\cap C(\mathbb{R}^n)$ is  a nonnegative distributional solution of \eqref{Har=01} in $\mathbb{R}^n$ with $1 < p < \frac{n-\alpha}{n-2\sigma}$, then by using \eqref{Glo-E002gyd} we get that $v\in \mathcal{L}_s (\mathbb{R}^n)$ for any $s > \frac{2\sigma - n}{2}$.   
\end{proof}

Based on Propositions \ref{Glo-P01}, \ref{Plo-P02} and  \ref{VVV},  we now show the integral representation of solutions to the higher order fractional equation \eqref{Har=01} stated in Theorem \ref{THM=02}.    Our proof is inspired by that of \cite[Theorem 1.8]{AGHW}, where the superharmonicity property for fractional Laplacian equations in $\mathbb{R}^n$ was obtained.   Similar idea has also been used in \cite{DY} for  the case when $\alpha=0$, $p=\frac{n+2\sigma}{n-2\sigma}$ and $\sigma \in (0, n/2)$ is an integer.  

\vskip0.10in

\noindent{\it Proof of Theorem \ref{THM=02} }.  Define $v$ as in \eqref{FS098}. Then, by Proposition \ref{VVV} we know that  $v\in \mathcal{L}_s (\mathbb{R}^n)$ for any $s > \frac{\alpha-2\sigma}{2(p-1)}$ if $p \geq  \frac{n-\alpha}{n-2\sigma}$ and $v\in \mathcal{L}_s (\mathbb{R}^n)$ for any $s > \frac{2\sigma - n}{2}$ if $1 < p <\frac{n-\alpha}{n-2\sigma}$.  In particular,  $v \in \mathcal{L}_0 (\mathbb{R}^n)$ for any case.   Moreover,  for any $\varphi \in C_c^\infty(\R^n)$ we have 
$$
\aligned
\int_{\R^n} v(x) (- \Delta)^\sigma \varphi(x) dx & = \int_{\R^n} \bigg( C_{n, \sigma} \int_{\R^n} \frac{|y|^{-\alpha}u(y)^p}{|x - y|^{n - 2 \sigma}} dy \bigg) (- \Delta)^\sigma \varphi(x) dx \\
& = \int_{\R^n} |y|^{-\alpha}u(y)^p  \bigg( C_{n, \sigma} \int_{\R^n} \frac{(- \Delta)^\sigma \varphi(x)}{|x - y|^{n - 2\sigma}} dx \bigg) dy \\
& = \int_{\R^n} |y|^{-\alpha}u(y)^p  \varphi(y) dy, 
\endaligned
$$
where the Fubini's theorem was used in the second equality.  Thus $v$ is a nonnegative distributional solution to 
\begin{equation}\label{Dis-V0}
(-\Delta)^\sigma v =|x|^{-\alpha} u^{p} ~~~~~~ \textmd{in} ~ \mathbb{R}^n. 
\end{equation} 
Let $w:=u-v$. Then,  by Proposition \ref{Glo-P01} we have  that $(-\Delta)^\sigma w =0$ on $\mathbb{R}^n$ in the distributional sense.  On the other hand, from Propositions \ref{Plo-P02} and  \ref{VVV} we obtain  that both $u$ and $v$ are in the space $\mathcal{L}_0(\mathbb{R}^n)$ and hence $w \in \mathcal{L}_0(\mathbb{R}^n)$.  It follows from a Liouville-type theorem (see, e.g.,  \cite{AGHW}) that $w \equiv 0$ on $\mathbb{R}^n$.  Thus, we have 
$$
u(x)=C_{n,\sigma} \int_{\mathbb{R}^n} \frac{|y|^{-\alpha}u(y)^{p}}{|x-y|^{n-2\sigma}} dy~~~~~ \textmd{for} ~  x \in \mathbb{R}^n \backslash \{0\}.  
$$
When $u \in C(\mathbb{R}^n)$, the above holds for all $x\in \mathbb{R}^n$.   Theorem \ref{THM=02} is established.     
\hfill$\square$

\section{Liouville-type theorems}\label{S2=0hy}
In this section, we first prove a Liouville-type theorem for nonnegative solutions of the following integral equation  
\begin{equation}\label{IntoLiou=001}
u(x)= \int_{\mathbb{R}^n} \frac{u(y)^{p}}{|y|^{\alpha} |x-y|^{n-2\sigma}} dy, ~~~~~~ x \in \mathbb{R}^n. 
\end{equation} 

\begin{theorem}\label{IntLiou=02App01}
Let $0< \sigma  < n/2$ and $-\infty <  \alpha < 2\sigma$.  Suppose that $u \in C(\mathbb{R}^n)$ is a nonnegative solution of \eqref{IntoLiou=001} with  $0 < p < \frac{n+2\sigma -2\alpha}{n-2\sigma}$.  Then $u \equiv 0$ in $\mathbb{R}^n$. 
\end{theorem}  

\br\label{LioutRem=001}
Such Liouville-type theorem for the integral equation \eqref{IntoLiou=001} has been proved by Li \cite{Li04} for $\alpha=0$,   by Dai-Qin \cite{DQ} for  $\sigma =m \in (1, n/2)$ being an  integer and $-\infty <  \alpha < 2m$, and by Cao-Dai-Qin\cite{CDQ} for non-integral $\sigma \in (1, n/2)$ and $-\infty <  \alpha \leq 0$.  Here we will prove Theorem \ref{IntLiou=02App01} by using the method of moving spheres in Li, Zhang and Zhu \cite{Li04,LZ03,Li-Zhu} and a ``Bootstrap" iteration argument in Dai-Qin \cite{DQ}.   
\er

For any $\lambda >0$ and nonnegative  function $u$,  we define the Kelvin transform of $u$ with respect to the ball $B_\lambda(0)$ by 
\begin{equation}\label{IntN0Ta-0022}
u_{\lambda}(\xi) = \left( \frac{\lambda}{|\xi|} \right)^{n-2\sigma} u(\xi^{\lambda}) ~~~~~~ \textmd{for}  ~ \xi\neq 0, 
\end{equation}
where $\xi^{\lambda}:= \frac{\lambda^2  \xi }{|\xi |^2}$.  It is easy to check that $(\xi^{\lambda})^{\lambda} =\xi$ and $(u_{\lambda})_{\lambda} \equiv u$.  By making a change of variables $y=z^{\lambda} = \frac{\lambda^2 z }{|z|^2}$,  we have 
$$
\aligned
\int_{|y| \geq \lambda} \frac{|y|^{-\alpha} u(y)^{p}}{|\xi^{\lambda} - y|^{n-2\sigma}} dy & = \int_{|z| \leq \lambda} \frac{|z^{ \lambda}|^{-\alpha} u(z^{ \lambda})^{p}}{|\xi^{\lambda} - z^{\lambda}|^{n-2\sigma}} \left( \frac{\lambda}{|z|} \right)^{2n} dz \\
& = \left( \frac{\lambda}{|\xi |} \right)^{-(n-2\sigma)} \int_{|z |\leq \lambda} \frac{|z|^{-\alpha} u_{\lambda}(z)^p}{|\xi - z|^{n-2\sigma}} \left( \frac{\lambda}{|z|} \right)^{n+2\sigma-2\alpha  - p(n-2\sigma)} dz, 
\endaligned
$$
where we have used the fact  that $\frac{|z|}{\lambda} \frac{|\xi |}{\lambda} |\xi^{\lambda} - z^{\lambda}| =  |\xi - z|$. Hence,  we obtain 
\begin{equation}\label{IntId-01}
\left( \frac{\lambda}{|\xi |} \right)^{n-2\sigma} \int_{|y |\geq \lambda}  \frac{|y|^{-\alpha} u(y)^{p}}{|\xi^{\lambda} - y|^{n-2\sigma}} dy  = \int_{|z |\leq \lambda} \frac{|z|^{-\alpha} u_{\lambda}(z)^p}{|\xi - z|^{n-2\sigma}} \left( \frac{\lambda}{|z|} \right)^{n+2\sigma -2\alpha - p(n-2\sigma)} dz. 
\end{equation}
Similarly, we also have 
\begin{equation}\label{IntId-02}
\left( \frac{\lambda}{|\xi |} \right)^{n-2\sigma} \int_{|y | \leq \lambda}  \frac{|y|^{-\alpha} u(y)^{p}}{|\xi^{\lambda} - y|^{n-2\sigma}} dy  = \int_{|z |\geq \lambda} \frac{|z|^{-\alpha} u_{\lambda}(z)^p}{|\xi - z|^{n-2\sigma}} \left( \frac{\lambda}{|z|} \right)^{n+2\sigma -2\alpha  - p(n-2\sigma)} dz.  \end{equation} 
Thus, if $u \in C(\mathbb{R}^n )$ is a nonnegative solution of \eqref{IntoLiou=001}, then  by \eqref{IntId-01} and \eqref{IntId-02} we get 
\begin{equation}\label{IntABC-01}
u_{\lambda}(\xi) = \int_{\mathbb{R}^n} \frac{|z|^{-\alpha} u_{\lambda}(z)^p}{|\xi - z|^{n-2\sigma}} \left( \frac{\lambda}{|z|} \right)^{n+2\sigma -2\alpha - p(n-2\sigma)} dz 
\end{equation}
and
\begin{equation}\label{IntABC-0hy2}
u(\xi) -u_{\lambda}(\xi) =\int_{|z|\geq \lambda} K(0, \lambda; \xi, z) \left[  \frac{u(z)^p}{|z|^{\alpha}}  - \left( \frac{\lambda}{|z|} \right)^{n+2\sigma -2\alpha  - p(n-2\sigma)} \frac{u_{\lambda}(z)^p}{|z|^{\alpha}}  \right] dz,   
\end{equation} 
where
$$
K(0, \lambda; \xi, z)=\frac{1}{|\xi -z|^{n-2\sigma}} - \left(\frac{\lambda}{|\xi |}  \right)^{n-2\sigma} \frac{1}{|\xi^{\lambda} -z|^{n-2\sigma}}. 
$$
It is elementary to check that   
$$
K(0, \lambda; \xi, z) >0~~~~~ \textmd{for} ~ \textmd{all} ~ |\xi|, |z| >\lambda>0.  
$$

To apply the moving sphere method, we first prove the following result which provides us a starting point.       

\begin{lemma}\label{S3oLeM=001}
Let   $0< \sigma  < n/2$ and $-\infty <  \alpha < 2\sigma$.  Suppose that  $u \in C(\mathbb{R}^n)$ is a positive solution of \eqref{IntoLiou=001} with  $0 < p < \frac{n+2\sigma -2\alpha}{n-2\sigma}$.   Then there exists a small  $\lambda_0 >0$ such that for any $0< \lambda < \lambda_0$  we have
\begin{equation}\label{S3o5SC101}
u_{\lambda} (y) \leq u(y) ~~~~~ \forall ~ |y| \geq \lambda. 
\end{equation}
\end{lemma} 
\begin{proof}
By Fatou lemma,  we have  
$$
 \liminf_{ |x| \to \infty }   |x|^{n-2\sigma} u(x) = \liminf_{ |x| \to \infty}   \int_{\mathbb{R}^n} \frac{|x|^{n-2\sigma} u(y)^{p} }{|y|^{\alpha} |x - y|^{n-2\sigma}} dy  \geq \int_{\mathbb{R}^n} \frac{u(y)^{p}}{|y|^\alpha} dy >0. 
$$
This, together with the positivity and continuity of $u$, yields that for each $r_0>0$
\be\label{CC=b0iei}
u(y) \geq \frac{C(r_0)}{|y|^{n-2\sigma}} ~~~~~ \forall ~ |y| \geq r_0 
\ee
with some constant $C(r_0) >0$. Hence, for every $r_0 >0$, there exists $\lambda_1= \lambda_1(r_0) \in (0, r_0)$ such that for any $0 < \lambda <\lambda_1$, 
\be\label{LilIn=001}
u_\lambda(y) =  \left( \frac{\lambda}{| y |} \right)^{n-2\sigma} u\left( \frac{\lambda^2 y}{|y|^2} \right) \leq \left( \frac{\lambda}{| y |} \right)^{n-2\sigma} \sup_{B_{r_0}} u \leq u(y)~~~~~ \forall ~ |y| \geq r_0.
\ee

Next we will show that $u_\lambda(y)  \leq u(y)$ on $B_{r_0} \backslash B_\lambda$ if $r_0$ is small and $0< \lambda < \lambda_1(r_0)$. Define $w^\lambda(y) = u(y) - u_\lambda(y)$ for $|y| \geq \lambda$ and denote 
$$
B_{\lambda}^-:= \{y \in \mathbb{R}^n \backslash B_\lambda: w^\lambda(y) <0 \}. 
$$
Then by \eqref{LilIn=001} we know that $B_{\lambda}^- \subset B_{r_0} \backslash B_\lambda$ for any fixed $r_0>0$ and for all $0< \lambda < \lambda_1(r_0)$. 
Using \eqref{IntABC-0hy2} we have that  for $0< \lambda < \lambda_1(r_0)$ and $y \in B_{\lambda}^-$,  
$$
\aligned
0> w^\lambda(y)  & = \int_{|z|\geq \lambda} K(0, \lambda; y, z) \left[  \frac{u(z)^p}{|z|^{\alpha}}  - \left( \frac{\lambda}{|z|} \right)^{n+2\sigma -2\alpha  - p(n-2\sigma)} \frac{u_{\lambda}(z)^p}{|z|^{\alpha}}  \right] dz \\
& \geq \int_{B_\lambda^-} K(0, \lambda; y, z) |z|^{-\alpha}  \left(  u(z)^p - u_\lambda(z)^p  \right) dz  \\
& \geq p \int_{B_\lambda^-} \frac{1}{|y -z|^{n-2\sigma}} |z|^{-\alpha} \max\{ u(z)^{p-1}, u_\lambda(z)^{p-1} \}   w^\lambda(z) dz. 
\endaligned
$$
Hence,  by Hardy-Littlewood-Sobolev inequality (see, e.g., \cite{Lieb}),  we obtain that for any $\frac{n}{n-2\sigma} < q < \infty$, 
\be\label{LilIn=002}
\aligned
\|w^\lambda\|_{L^q(B_\lambda^-)} & \leq C  \big\| |\cdot|^{-\alpha} \max\{ u^{p-1}, u_\lambda^{p-1} \} w^\lambda \big\|_{L^{\frac{nq}{n+2\sigma q}}(B_\lambda^-)} \\
& \leq C \big\| |\cdot|^{-\alpha} \max\{ u^{p-1}, u_\lambda^{p-1} \} \big\|_{L^{\frac{n}{2\sigma}}(B_\lambda^-)}  \|w^\lambda\|_{L^q(B_\lambda^-)}. 
\endaligned 
\ee
Since $\|\max\{ u^{p-1}, u_\lambda^{p-1}\} \|_{L^\infty(B_\lambda^-)} \leq \|u^{p-1}\|_{L^\infty(B_{r_0})}$ and $\alpha< 2\sigma$, we can take $r_0=\varepsilon_0$ for some $\varepsilon_0>0$ small enough such that 
$$
C \big\| |\cdot|^{-\alpha} \max\{ u^{p-1}, u_\lambda^{p-1} \} \big\|_{L^{\frac{n}{2\sigma}}(B_\lambda^-)} \leq C \big\| |\cdot|^{-\alpha} \big\|_{L^{\frac{n}{2\sigma}}(B_{\varepsilon_0})} \big\|u^{p-1} \big\|_{L^\infty(B_{\varepsilon_0})}  \leq \frac{1}{2} 
$$
for all $0< \lambda < \lambda_0:=\lambda_1(\varepsilon_0)$,  and thus, \eqref{LilIn=002} implies that 
$$
\|w^\lambda\|_{L^q(B_\lambda^-)}  =0. 
$$
Therefore, $B_\lambda^- =\emptyset$ for any $0< \lambda < \lambda_0$. That is, \eqref{S3o5SC101} holds for any $0< \lambda < \lambda_0$. 
\end{proof} 

We define 
$$
\bar{\lambda}:= \sup\{ \mu > 0 ~ | ~ u_{\lambda}(y) \leq u(y) ~ \forall ~ |y | \geq \lambda, ~ \forall ~ 0< \lambda <\mu \}. 
$$
By Lemma \ref{S3oLeM=001}, $\bar{\lambda}$ is well-defined and $\bar{\lambda}>0$. Moreover, we have the following result.    

\begin{lemma}\label{S3InLem02}
Let   $0< \sigma  < n/2$ and $-\infty <  \alpha < 2\sigma$.  Suppose that  $u \in C(\mathbb{R}^n)$ is a positive solution of \eqref{IntoLiou=001} with  $0 < p < \frac{n+2\sigma -2\alpha}{n-2\sigma}$.   Then $\bar{\lambda} = \infty$. 
\end{lemma} 
\begin{proof}
Suppose by contradiction that $\bar{\lambda} < \infty$. The definition of $\bar{\lambda}$  implies  that 
\be\label{S3InLem02=001} 
u_{\bar{\lambda}} (y)\leq u(y) ~~~~~ \forall ~ |y | \geq \bar{\lambda}.
\ee
Since $n+2\sigma-2\alpha - p(n-2\sigma) >0$, we have $\left( \frac{\bar{\lambda}}{|z|} \right)^{n+2\sigma-2\alpha - p(n-2\sigma)} <1$  for $|z| > \bar{\lambda}$. Using \eqref{IntABC-0hy2} we get, for $|y| > \bar{\lambda}$, 
$$
\aligned
u(y) -u_{\bar{\lambda}}(y) & = \int_{|z|\geq \bar{\lambda}} K(0, \bar{\lambda}; y, z) \left[  \frac{u(z)^p}{|z|^{\alpha}}  - \left( \frac{\bar{\lambda}}{|z|} \right)^{n+2\sigma -2\alpha  - p(n-2\sigma)} \frac{u_{\bar{\lambda}}(z)^p}{|z|^{\alpha}}  \right] dz  \\
& \geq \int_{|z|\geq \bar{\lambda}} K(0, \bar{\lambda}; y, z) \left[  1  - \left( \frac{\bar{\lambda}}{|z|} \right)^{n+2\sigma -2\alpha  - p(n-2\sigma)} \right] \frac{u_{\bar{\lambda}}(z)^p}{|z|^{\alpha}}   dz >0. 
\endaligned
$$
This, together with the Fatou lemma, yields 
$$
\aligned
\liminf_{ |y| \to \infty } |y|^{n-2\sigma} (u -u_{\bar{\lambda}})(y) & \geq  \liminf_{ |y| \to \infty } \int_{|z|\geq \bar{\lambda}}  |y|^{n-2\sigma}K(0, \bar{\lambda}; y, z) \frac{u(z)^p - u_{\bar{\lambda}}(z)^p}{|z|^{\alpha}}  dz\\
& \geq \int_{|z|\geq \bar{\lambda}}  \left[  1  - \left( \frac{\bar{\lambda}}{|z|} \right)^{n-2\sigma} \right] \frac{u(z)^p - u_{\bar{\lambda}}(z)^p}{|z|^{\alpha}}  dz >0. 
\endaligned
$$
Consequently, for $\delta>0$ small which will be fixed later, there exists $c_1=c_1(\delta) > 0$ such that 
$$
u(y)-u_{\bar{\lambda}}(y) \geq \frac{c_1}{|y|^{n-2\sigma}}  ~~~~~ \forall ~ |y | \geq \bar{\lambda} + \delta. 
$$
By the above and explicit formula of $u_\lambda$, there exists a small $\varepsilon \in (0, \delta)$ such that for all $\bar{\lambda} \leq \lambda \leq \bar{\lambda} + \varepsilon$, 
\be\label{S3InLem02=002} 
u(y) - u_\lambda(y) \geq \frac{c_1}{|y|^{n-2\sigma}} + \big( u_{\bar{\lambda}}(y) - u_\lambda(y) \big) \geq \frac{c_1}{2|y|^{n-2\sigma}} ~~~~ \forall ~ |y | \geq \bar{\lambda} +\delta. 
\ee
Now,  we focus on the narrow region $B_{\bar{\lambda} +\delta} \backslash B_\lambda$.  As in the proof of Lemma \ref{S3oLeM=001},   we define $w^\lambda(y) = u(y) - u_\lambda(y)$ for $|y| \geq \lambda$ and denote  
$$
B_{\lambda}^-:= \{y \in \mathbb{R}^n \backslash B_\lambda: w^\lambda(y) <0 \}. 
$$
By \eqref{S3InLem02=002} we have $B_{\lambda}^- \subset B_{\bar{\lambda} +\delta} \backslash B_\lambda$ for any $\bar{\lambda} \leq \lambda \leq \bar{\lambda} + \varepsilon$.  Thus,  using  \eqref{IntABC-0hy2}  we have,  for $y \in B_{\lambda}^-$ and for $\bar{\lambda} \leq \lambda \leq \bar{\lambda} +\varepsilon$,  
$$
\aligned
0 > w^\lambda(y)  & \geq  \int_{|z|\geq \lambda} K(0, \lambda; y, z) |z|^{-\alpha}  \left( u(z)^p - u_{\lambda}(z)^p \right)  dz  \\
& \geq p \int_{B_\lambda^-} \frac{1}{|y -z|^{n-2\sigma}} |z|^{-\alpha} \max\{ u(z)^{p-1}, u_\lambda(z)^{p-1} \}   w^\lambda(z) dz.
\endaligned
$$
By Hardy-Littlewood-Sobolev inequality (see, e.g., \cite{Lieb}),  we obtain that for any $\frac{n}{n-2\sigma} < q < \infty$, 
\be\label{S3InLem02=003}
\aligned
\|w^\lambda\|_{L^q(B_\lambda^-)} & \leq C  \big\| |\cdot|^{-\alpha} \max\{ u^{p-1}, u_\lambda^{p-1} \} w^\lambda \big\|_{L^{\frac{nq}{n+2\sigma q}}(B_\lambda^-)} \\
& \leq C \big\| |\cdot|^{-\alpha} \max\{ u^{p-1}, u_\lambda^{p-1} \} \big\|_{L^{\frac{n}{2\sigma}}(B_\lambda^-)}  \|w^\lambda\|_{L^q(B_\lambda^-)}. 
\endaligned 
\ee
Note that we can choose $\delta>0$ small enough such that 
$$
C \big\| |\cdot|^{-\alpha} \max\{ u^{p-1}, u_\lambda^{p-1} \} \big\|_{L^{\frac{n}{2\sigma}}(B_\lambda^-)} \leq C \big\| |\cdot|^{-\alpha} \max\{ u^{p-1}, u_\lambda^{p-1} \} \big\|_{L^{\frac{n}{2\sigma}}(B_{\bar{\lambda} +\delta} \backslash B_\lambda)}  \leq \frac{1}{2} 
$$
for all $\bar{\lambda} \leq \lambda \leq  \bar{\lambda} +\varepsilon$,  and thus, \eqref{S3InLem02=003} implies that  
$$
\|w^\lambda\|_{L^q(B_\lambda^-)}  =0. 
$$
Therefore, $B_\lambda^- =\emptyset$ for all $\bar{\lambda} \leq \lambda \leq  \bar{\lambda} +\varepsilon$.  That is, there exists $\varepsilon>0$ such that for all $\bar{\lambda} \leq \lambda \leq  \bar{\lambda} +\varepsilon$, 
$$
u_\lambda(y) \leq u(y) ~~~~~ \forall ~ |y | \geq \lambda,
$$
which contradicts the definition of $\bar{\lambda}$.   
\end{proof}

\noindent{\it Proof of Theorem \ref{IntLiou=02App01}}.  
Suppose by contrary that  $u \in C(\mathbb{R}^n)$ is  a  nontrivial nonnegative solution of \eqref{IntoLiou=001},  then there exists $x_0 \in \mathbb{R}^n$ such that $u(x_0) >0$.  It follows from the equation \eqref{IntoLiou=001} that $u(x) >0$ for all $x \in \mathbb{R}^n$. Thus, $u \in C(\mathbb{R}^n)$ is  a positive solution of \eqref{IntoLiou=001}.   By Lemma \ref{S3InLem02} we have 
\be\label{S3InLem02=004}
u(x) \geq  \left( \frac{\lambda}{|x|}\right)^{n-2\sigma} u \left( \frac{\lambda^2 x}{|x|^2}  \right) ~~~~~ \forall ~ |x | \geq \lambda, ~ \forall ~ 0< \lambda < \infty. 
\ee
For any $|x| \geq 1$,  taking $\lambda = \sqrt{|x|}$  in \eqref{S3InLem02=004} yields that  
\be\label{S3InLem02=005}
u(x) \geq |x|^{-\frac{n-2\sigma}{2}} u\left( \frac{x}{|x|}  \right).
\ee
Hence, we have the following lower bound:
\be\label{S3InLem02=006}
u(x) \geq \left( \min_{\partial B_1} u \right) |x|^{-\frac{n-2\sigma}{2}} =: C_0  |x|^{-\frac{n-2\sigma}{2}} ~~~~~ \forall ~ |x | \geq  1. 
\ee
Now, we use the `` Bootstrap" iteration argument to improve the lower bound \eqref{S3InLem02=006}.  Let $\mu_0:=\frac{n-2\sigma}{2}$.  From \eqref{IntoLiou=001} and \eqref{S3InLem02=006} we have,  for $|x| \geq 1$, 
$$
\aligned 
u(x) &\geq C \int_{|x| \leq |y| \leq 2|x|} \frac{1}{|x -y|^{n-2\sigma} |y|^{p\mu_0 + \alpha}} dy \\
& \geq  C |x|^{-(n-2\sigma)} \int_{|x| \leq |y| \leq 2|x|} \frac{1}{|y|^{p\mu_0 + \alpha}} dy \\
& =: C_1 |x|^{-(p \mu_0 + \alpha -2\sigma)}
\endaligned
$$
with some constant $C_1>0$.  Let $\mu_1:=p \mu_0 + \alpha -2\sigma$. Since $0 < p < \frac{n+2\sigma-2\alpha}{n-2\sigma}$, we have
\be\label{S3InLem02=00iig}
\mu_1=p \mu_0 + \alpha -2\sigma < \mu_0. 
\ee
This means that we have obtained a better lower bound after one iteration,  
\be\label{S3InLem02=00ff}
u(x) \geq  C_1 |x|^{-\mu_1}  ~~~~~ \forall ~ |x | \geq  1. 
\ee
Continuing this process, after $k$ iterations we would get the following lower bound 
\be\label{S3InLem02=0fhry}
u(x) \geq  C_k |x|^{-\mu_k}  ~~~~~ \forall ~ |x | \geq  1, 
\ee
where $C_k>0$ is a constant and the sequence $\{\mu_k\}_{k\geq 1}$ satisfies 
\be\label{S3InLem02=00gmfhy}
\mu_{k}= p \mu_{k-1} +\alpha -2\sigma,~~~~~~~ k=1, 2, \cdots.   
\ee
It is easy to check that the sequence $\{\mu_k\}_{k\geq 1}$ is monotonically decreasing with respect to $k$.  Moreover,  as $k \to \infty$, 
$$
\mu_k \to 
\begin{cases}
\frac{\alpha - 2\sigma}{1-p} ~~~~ & \textmd{if} ~ 0 < p< 1, \\
-\infty ~~~~ & \textmd{if} ~ 1 \leq p< \frac{n+2\sigma-2\alpha}{n-2\sigma}. 
\end{cases}
$$
Thus, by \eqref{S3InLem02=0fhry} we have the following lower estimates: for $|x| \geq 1$,  
$$
u(x) \geq 
\begin{cases}
C_\kappa|x|^\kappa ~~~~ \forall ~ \kappa < \frac{2\sigma - \alpha}{1-p},  ~~~~ & \textmd{if} ~ 0 < p< 1, \\
C_\kappa|x|^\kappa ~~~~ \forall ~ \kappa < +\infty,   ~~~~ & \textmd{if} ~ 1 \leq p< \frac{n+2\sigma-2\alpha}{n-2\sigma}. 
\end{cases}
$$
It follows from the equation \eqref{IntoLiou=001} and  above estimates that  
$$
u(0) = \int_{\mathbb{R}^n} \frac{u(y)^{p}}{|y|^{n+\alpha-2\sigma}} dy \geq \int_{|y| \geq 1} \frac{1}{|y|^{n+\alpha-2\sigma}} dy =+\infty, 
$$
which is a contradiction. Therefore, we must have $u \equiv 0$ in $\mathbb{R}^n$.  The proof of Theorem \ref{IntLiou=02App01} is completed. 
\hfill$\square$

\vskip0.10in 

\noindent{\it Proof of Theorem \ref{THM=02App02}.} It follows from Theorems \ref{THM=02} and  \ref{IntLiou=02App01}. 
\hfill$\square$

\section{Existence of solutions}\label{Ex}
In this section, we will prove Theorem \ref{THMExi01} using the classical bifurcation theory and a higher order fractional Pohozaev identity established by Ros-Oton and Serra \cite{RS1,RS2}. The similar method was recently used by Ao-Chan-DelaTorre-Fontelos-Gonz\'{a}lez-Wei \cite{AW1}, Hyder-Sire \cite{HS}, and Ng\^{o}-Ye \cite{NY} for fractional Laplacian ($0 < \sigma <1$), biharmonic case ($\sigma=2$), and polyharmonic case ($\sigma$ is a positive integer), respectively.

We first consider the following auxiliary problem 
\begin{equation}\label{24=fra-au01}
\begin{cases}
(-\Delta)^\sigma u= \lambda |x|^{-\alpha} (1 + u)^p ~~~~ & \textmd{in} ~ B_1, \\
u \equiv 0~~~~ &\textmd{in} ~ \mathbb{R}^n \setminus B_1, 
\end{cases}
\end{equation}
where $\lambda >0$. Under the Dirichlet boundary condition, the Green function $G_\sigma(x, y)$ of $(-\Delta)^\sigma$ in $B_1$ is positive. Indeed, it follows from Dipierro-Grunau \cite{DG} that the classical Boggio formula 
\begin{equation}\label{Boggio}
G_\sigma(x, y) = k_{\sigma, n} |x-y|^{2\sigma -n} \int_1^{\frac{\sqrt{|x|^2|y|^2 -2 x \cdot y +1}}{|x-y|}} (t^2 -1)^{\sigma -1} t^{1-n} dt ~~~~~ \textmd{for} ~ x, y \in B_1, 
\end{equation}
holds true for any $\sigma>0$. From the positivity of $G_\sigma$, one can follow the standard monotone iteration argument (see, e.g., \cite{RS3,NY}) to obtain that there exists a minimal solution $u_\lambda$ for small $\lambda$. Moreover, one can find a $\lambda^*>0$ such that 
\begin{enumerate}[label = \rm(\roman*)]
\item the minimal solution $u_\lambda$ exists for any $\lambda \in (0, \lambda^*)$, $u_\lambda$ is radially symmetric and nonincreasing, and the mapping $\lambda \mapsto u_\lambda(x)$ is nondecreasing in $(0, \lambda^*)$ for any $x \in B_1$; 
 
\item for $\lambda > \lambda^*$, \eqref{24=fra-au01} has no solutions. 
\end{enumerate} 
 
Next, we show the uniqueness of minimal solutions of \eqref{24=fra-au01} for small $\lambda>0$ using Schaaf's argument \cite{Sch}. This is based on the following higher order fractional Pohozaev identity obtained by Ros-Oton and Serra \cite{RS1,RS2}.

\begin{proposition}\label{Poho} 
Let $\Omega$ be a bounded smooth domain, let $f \in C_{loc}^{0,1}(\overline{\Omega} \times \mathbb{R})$, let $u$ be a bounded solution of 
\begin{equation}\label{24=fra-002}
\begin{cases}
(-\Delta)^\sigma u= f(x, u) ~~~~ & \textmd{in} ~ \Omega, \\
u \equiv 0~~~~ &\textmd{in} ~ \mathbb{R}^n \setminus \Omega, 
\end{cases}
\end{equation}
and let $\delta(x)=\mbox{dist}(x, \partial \Omega)$. Then $u/\delta^\sigma |_{\Omega}$ has a continuous extension to $\overline{\Omega}$, and the following identity holds 
$$
\int_\Omega \left( F(x, u) + \frac{1}{n} x \cdot \nabla_x F(x, u) - \frac{n-2\sigma}{2n} uf(x, u) \right) dx = \frac{\Gamma(1+\sigma)^2}{2n} \int_{\partial\Omega} \left( \frac{u}{\delta^\sigma} \right)^2 (x \cdot \nu) dS, 
$$
where $F(x,t)=\int_0^t f(x, \tau) d\tau$, $\nu$ is the unit outward normal to $\partial\Omega$ at $x$, and $\Gamma$ is the Gamma function. 
 \end{proposition}
 
Using integration by parts, we also have 
\[
\int_\Omega uf(x, u) dx =\int_{\mathbb{R}^n} |(-\Delta)^\frac{\sigma}{2} u|^2 dx. 
\]
Thus, the above Pohozaev identity leads to a fundamental inequality as follows. 
\begin{corollary}\label{Poho=co01} 
Under the assumptions of Proposition \ref{Poho}, for any star-shaped domain $\Omega$ and any $\kappa \in \mathbb{R}$, we have 
$$
\int_\Omega \left( F(x, u) + \frac{1}{n} x \cdot \nabla_x F(x, u) - \kappa uf(x, u) \right) dx \geq  \left(\frac{n-2\sigma}{2n} - \kappa \right) \int_{\mathbb{R}^n} |(-\Delta)^\frac{\sigma}{2} u|^2 dx. 
$$ 
\end{corollary} 

Now we are in a position to prove the uniqueness of solutions of \eqref{24=fra-au01} for all small $\lambda>0$. Remark that the supercritical growth is crucial here. 
\begin{lemma}\label{unique=24}
Let $p>p_{\sigma,\alpha}^*$. Then there exists $\lambda_0>0$ such that for every $\lambda \in (0, \lambda_0)$, $u_\lambda$ is the unique solution of \eqref{24=fra-au01} among the class
\[
\mathcal{H}^\sigma_c(\mathbb{R}^n) = \{u \in H^\sigma(\mathbb{R}^n) \cap C(\mathbb{R}^n): u \equiv 0 ~ in ~ \mathbb{R}^n \setminus B_1 \}. 
\]
\end{lemma}
\begin{proof}
Let $u_\lambda$ be the minimal solution of \eqref{24=fra-au01}, and let $u$ be another solution of \eqref{24=fra-au01}. Then $w:=u-u_\lambda \geq 0$ in $B_1$ and $w$ is a solution of 
\begin{equation}\label{unique=24=001}
\begin{cases}
(-\Delta)^\sigma w= \lambda |x|^{-\alpha} g_\lambda(x, w) ~~~~ & \textmd{in} ~ B_1, \\
w \equiv 0~~~~ &\textmd{in} ~ \mathbb{R}^n \setminus B_1, 
\end{cases}
\end{equation}
where $g_\lambda(x, w) = (1 + w + u_\lambda(x))^p - (1+u_\lambda(x))^p \geq 0$. We want to show that $w \equiv 0$ for small $\lambda>0$. Denoting 
$$
G_\lambda(x, w) =\int_0^w g_\lambda(x, t) dt, 
$$
and using Corollary \ref{Poho=co01} in $B_1$ we obtain, for any $\kappa \in \mathbb{R}$, that 
\begin{equation}\label{unique=24=002}
\aligned 
& \Big(\frac{n-2\sigma}{2n} - \kappa \Big) \int_{\mathbb{R}^n} |(-\Delta)^\frac{\sigma}{2} w|^2 dx \\
& ~~~ \leq \lambda \int_{B_1} \Big( |x|^{-\alpha} G_\lambda(x, w) + \frac{1}{n} x \cdot \nabla_x \big( |x|^{-\alpha} G_\lambda(x, w) \big)  - \kappa |x|^{-\alpha} w g_\lambda(x, w) \Big) dx \\ 
& ~~~ = \lambda \int_{B_1} |x|^{-\alpha} \Big( \big(1 - \frac{\alpha}{n} \big) G_\lambda(x, w) + \frac{1}{n} x \cdot \nabla_x G_\lambda(x, w)  - \kappa w g_\lambda(x, w) \Big) dx. 
\endaligned 
\end{equation}
Note that 
\begin{equation}\label{unique=24=003}
\aligned 
G_\lambda(x, w) &= w \int_0^1 \big[ (1 + sw + u_\lambda(x))^p - (1+u_\lambda(x))^p  \big]ds \\
&=p w^2 \int_0^1 \int_0^1 s(1 + u_\lambda(x) + \tau s w)^{p-1} d\tau ds. 
\endaligned 
\end{equation}
Hence
$$
\nabla_x G_\lambda(x, w) = \Big[ p(p-1) w^2 \int_0^1 \int_0^1 s(1 + u_\lambda(x) + \tau s w)^{p-2} d\tau ds \Big] \nabla u_\lambda(x). 
$$
Since $u_\lambda$ is radially decreasing, we have $x \cdot \nabla u_\lambda(x) \leq 0$, and then $x \cdot \nabla_x G_\lambda(x, w)  \leq 0$. Thus, by \eqref{unique=24=002} we have 
\begin{equation}\label{unique=24=004}
\Big(\frac{n-2\sigma}{2n} - \kappa \Big) \int_{\mathbb{R}^n} |(-\Delta)^\frac{\sigma}{2} w|^2 dx \leq  \lambda \int_{B_1} |x|^{-\alpha} \Big( \big(1 - \frac{\alpha}{n} \big) G_\lambda(x, w) - \kappa w g_\lambda(x, w) \Big) dx. 
\end{equation}
Recall that $0 \leq u_\lambda(x) \leq \|u_{\lambda^*/2}\|_{L^\infty(B_1)}$ for any $x \in B_1$ and any $\lambda \in [0, \lambda^*/2]$. Therefore, a direct calculation gives  
$$
\aligned 
\lim_{t \to +\infty} \frac{G_\lambda(x, t)}{ t g_\lambda(x,t) } &= \lim_{t \to +\infty} \frac{\frac{1}{p+1} \big[ (1+t+u_\lambda(x))^{p+1} - (1+ u_\lambda(x))^{p+1} \big] - t (1+ u_\lambda(x))^p }{t\big[ (1+t+u_\lambda(x))^{p} - (1+ u_\lambda(x))^{p} \big]} \\
& = \frac{1}{p+1} ~~~~ \mbox{uniformly for} ~ x \in B_1 ~ \mbox{and} ~ \lambda \in [0, \lambda^*/2]. 
\endaligned  
$$
This, together with \eqref{unique=24=003}, implies that for any $\epsilon >0$ there exists $M=M(\epsilon, p, \|u_{\lambda^*/2}\|_{L^\infty(B_1)}) >0$ such that 
\begin{equation}\label{unique=24=005}
G_\lambda(x, t) \leq \frac{1+\epsilon}{p+1} t g_\lambda(x,t) + M t^2
\end{equation}
for all $(x, t, \lambda) \in B_1 \times \mathbb{R}_+ \times [0, \lambda^*/2]$. Moreover, we observe that 
$$
\Big(1 -\frac{\alpha}{n} \Big) \frac{1}{p+1} - \frac{n-2\sigma}{2n} = \frac{(n+2\sigma -2\alpha) -p(n-2\sigma)}{2n(p+1)} <0 
$$
as $p>\frac{n+2\sigma -2\alpha}{n-2\sigma}$. Hence, one can choose $\epsilon >0$ and $\kappa>0$ such that 
$$
\Big(1 -\frac{\alpha}{n} \Big) \frac{1+\epsilon}{p+1} < \kappa < \frac{n-2\sigma}{2n}. 
$$
With these choices, by \eqref{unique=24=004} and \eqref{unique=24=005} we obtain 
\begin{equation}\label{unique=24=006}
\Big(\frac{n-2\sigma}{2n} - \kappa \Big) \int_{\mathbb{R}^n} |(-\Delta)^\frac{\sigma}{2} w|^2 dx \leq  \lambda M \Big(1 - \frac{\alpha}{n} \Big) \int_{B_1} |x|^{-\alpha} w^2 dx. 
\end{equation}
On the other hand, using the fractional Hardy-Sobolev inequality and the H\"{o}lder inequality we know 
$$
\int_{B_1} |x|^{-\alpha} w^2 dx \leq C_{n,\alpha,\sigma} \int_{\mathbb{R}^n} |(-\Delta)^\frac{\sigma}{2} w|^2 dx 
$$
for some constant $C_{n,\alpha,\sigma} >0$. Therefore, we have 
$$
\int_{\mathbb{R}^n} |(-\Delta)^\frac{\sigma}{2} w|^2 dx \leq \lambda M C_{n,\alpha,\sigma} \Big(1 - \frac{\alpha}{n} \Big) \Big(\frac{n-2\sigma}{2n} - \kappa \Big)^{-1} \int_{\mathbb{R}^n} |(-\Delta)^\frac{\sigma}{2} w|^2 dx, 
$$
which yields $w \equiv 0$ if
$$
\lambda < \lambda_0:= \frac{\Big(\frac{n-2\sigma}{2n} - \kappa \Big)}{M C_{n,\alpha,\sigma} \Big(1 - \frac{\alpha}{n} \Big)}. 
$$
This completes the proof of Lemma \ref{unique=24}. 
\end{proof}
 
We also need to show the following result by applying the classical bifurcation theory. Denote 
$$
E =\{u \in C(\mathbb{R}^n) : u(x) = \tilde{u}(|x|), ~ u \geq 0 ~ \mbox{in} ~B_1 ~\mbox{and} ~ u \equiv 0 ~ \mbox{in} ~ \mathbb{R}^n \setminus B_1\}. 
$$
\begin{lemma}\label{bifur=24}
There exists a sequence of solutions $(\lambda_j, u_j)$ of \eqref{24=fra-au01} in $(0, \lambda^*] \times E$ such that
$$
\lim_{j \to \infty} \lambda_j = \lambda_\infty \in [\lambda_0, \lambda^*] ~~~~~~ and ~~~~~~ \lim_{j \to \infty} \|u_j\|_{L^\infty(B_1)} = \infty, 
$$
where $\lambda_0 >0$ is given in Lemma \ref{unique=24}.   
\end{lemma} 
\begin{proof}
Define the operator $\mathcal{T}: \mathbb{R}_+ \times E \to E$, $(\lambda, u) \mapsto \mathcal{T}(\lambda, u)$ as 
$$
\mathcal{T}(\lambda, u)(x) =\lambda \int_{B_1} G_\sigma(x, y) |y|^{-\alpha} (1 + u(y))^p dy, ~~~~~~~ x \in B_1, 
$$ 
where $G_\sigma$ is the Green function given in \eqref{Boggio}. Then $\mathcal{T}$ is compact, that is, it maps bounded sets to relatively compact sets. 
For each $\lambda \in (0, \lambda^*)$, any solution $u$ of \eqref{24=fra-au01} also satisfies 
\be\label{T=equ}
u=\mathcal{T}(\lambda, u). 
\ee
Consider the continuation
$$
\mathcal{C} =\{(\lambda(t), u(t)) : t \geq 0 \}
$$
of the branch of minimal solutions $\{(\lambda, u_\lambda) : 0 < \lambda \leq \lambda_0 \}$, where $(\lambda(0), u(0))=(\lambda_0, u_{\lambda_0})$. 
By Lemma \ref{unique=24} and the fact that no solutions exist for $\lambda > \lambda^*$, we have $\mathcal{C} \subset [\lambda_0, \lambda^*] \times E$. It follows from the classical bifurcation theory (see Rabinowitz \cite[Theorem 6.2]{Ra}) that $\mathcal{C}$ is unbounded in $[\lambda_0, \lambda^*] \times E$. Thus, we can extract a desired sequence of pairs $(\lambda_j, u_j)$. 
\end{proof}

We are now ready to prove Theorem \ref{THMExi01}. 

\vskip0.10in  

\noindent{\it Proof of Theorem \ref{THMExi01}. }
Let $(\lambda_j, u_j)$ be the sequence in Lemma \ref{bifur=24}. Define
\[
m_j= \|u_j\|_{L^\infty(B_1)}=u_j(0)  ~~~~ \mbox{and} ~~~~ r_j= (m_j^{p-1} \lambda_j)^{\frac{1}{2\sigma - \alpha} }. 
\]
such that $ m_j, r_j \to \infty $ as $j \to \infty$. Set
\[
w_j(x) = m_j^{-1} u_j \left( \frac{x}{r_j} \right).
\]
Then $w_j(0)=1$, $0 \leq w_j \leq 1$ in $\mathbb{R}^n$ and $w_j$ satisfies 
\begin{equation}\label{24=fra=wjj}
\begin{cases}
(-\Delta)^\sigma w_j = |x|^{-\alpha} (m_j^{-1} + w_j)^p ~~~~ & \textmd{in} ~ B_{r_j}, \\
w_j \equiv 0~~~~ &\textmd{in} ~ \mathbb{R}^n \setminus B_{r_j}. 
\end{cases}
\end{equation}
Since $\alpha < 2\sigma$, we have $|x|^{-\alpha} \in L_{\textmd{loc}}^q (\mathbb{R}^n) $ for some $q > \frac{n}{2\sigma}$. By the regularity results in \cite{G,JLX17,RS3}, we know that $w_j \in C_{\textmd{loc}}^\gamma(\mathbb{R}^n)$ for some $\gamma \in (0, 1)$. Hence, after passing to a subsequence, there exists a nonnegative function $w \in C(\mathbb{R}^n)$ such that $w_j$ converges to $w$ locally uniformly in $\mathbb{R}^n$. Then, $w(0)=1$, $0 \leq w \leq 1$ and $w$ is radially symmetric and nonincreasing. Moreover, $w$ is a distributional solution of  
$$
(-\Delta)^\sigma w = |x|^{-\alpha}w^{p} ~~~~~~ \textmd{in} ~ \mathbb{R}^n. 
$$
The proof of Theorem \ref{THMExi01} is completed.   
\hfill$\square$ 

\vskip0.10in  

Finally, we apply the Kelvin transform and Theorem \ref{THMExi01} to prove Corollary \ref{24=La-Em-002}. 

\vskip0.10in  

\noindent{\it Proof of Corollary \ref{24=La-Em-002}. } 
Suppose $\frac{n}{n-2\sigma} < p < \frac{n + 2\sigma}{n - 2\sigma}$. Let 
\[
\beta = (n+2\sigma) - p(n-2\sigma) \in (0, 2\sigma), 
\]
and let $w$ be the positive and radially decreasing solution of $(-\Delta)^\sigma w = |x|^{-\beta}w^{p}$ obtained in Theorem \ref{THMExi01}.  
By Theorem \ref{THM=02}, $w$ also satisfies the integral equation 
\[
w(x)= C_{n,\sigma} \int_{\mathbb{R}^n} \frac{w(y)^{p}}{|y|^{\beta} |x-y|^{n-2\sigma}} dy~~~~~ \textmd{for} ~ x \in \mathbb{R}^n. 
\]
Hence, for any $r>0$ and $\theta \in \mathbb{S}^{n-1}$ we have  
$$
\aligned
w(r \theta) & \geq  C_{n,\sigma}  \int_{B_r} \frac{w(y)^p}{|y|^\beta |r\theta -y|^{n-2\sigma}} dy\\
& = C_{n,\sigma} w(r\theta)^{p} r^{2\sigma - \beta} \int_0^1 \left( \int_{ \mathbb{S}^{n-1} }  \frac{1}{|\theta -t \eta|^{n-2\sigma}} d\eta \right) t^{n-\beta-1} dt\\
&= C_{n,\sigma,p} w(r\theta)^{p} r^{2\sigma - \beta}
\endaligned
$$
for some constant $C_{n,\sigma,p} >0$. This implies that
\begin{equation}\label{24=Kel=ot97}
w(x) \leq C |x|^{\frac{\beta-2\sigma}{p-1}} = C |x|^{\frac{2\sigma}{p-1} - (n-2\sigma)} ~~~~~~ \textmd{for} ~ x \in \mathbb{R}^n \setminus \{0\}.  
\end{equation}
Remark that \eqref{24=Kel=ot97} can also be obtained directly from uniform estimates of \eqref{24=fra-au01} (see, e.g., \cite[Lemma 2.6]{AW1}).  
For any $\varepsilon>0$, we consider the rescaled function $w_\varepsilon(x) = \varepsilon w (\varepsilon^{\frac{p-1}{2\sigma-\beta}} x)$, which solves  
$$
\begin{cases}
(-\Delta)^\sigma w_\varepsilon = |x|^{-\beta} w_\varepsilon^{p} ~~~~~~  & \textmd{in} ~ \mathbb{R}^n, \\
w_\varepsilon(0)=\varepsilon, \\
w_\varepsilon(x) \leq C |x|^{\frac{2\sigma}{p-1} - (n-2\sigma)} ~~~~~~  & \textmd{in} ~ \mathbb{R}^n \setminus \{0\}. 
\end{cases}
$$
Then, its Kelvin transform $u_\varepsilon(x) = |x|^{-(n-2\sigma)} w_\varepsilon \big( \frac{x}{|x|^2} \big)$ satisfies
$$
\begin{cases}
(-\Delta)^\sigma u_\varepsilon =  u_\varepsilon^{p} ~~~~~~  & \textmd{in} ~ \mathbb{R}^n \setminus \{0\}, \\
u_\varepsilon(x) \leq C |x|^{-\frac{2\sigma}{p-1}} ~~~~~~  & \textmd{in} ~ \mathbb{R}^n \setminus \{0\}, \\
u_\varepsilon(x) \sim \varepsilon |x|^{- (n-2\sigma)} ~~~~~~  & \textmd{as} ~ |x| \to \infty. 
\end{cases}
$$
This finishes the proof. 
\hfill$\square$

\section{Radial symmetry of solutions}\label{S3}
In this section,  we first prove the following radial symmetry of nonnegative solutions to the integral equation 
\begin{equation}\label{Int=21xb0r1} 
u(x)= \int_{\mathbb{R}^n} \frac{u(y)^{p}}{|y|^{\alpha} |x-y|^{n-2\sigma}} dy, ~~~~~~ x \in \mathbb{R}^n \backslash \{0\}. 
\end{equation} 
Then we have  
\begin{theorem}\label{THM=03}
Let $0< \sigma  < n/2$ and $0\leq \alpha < 2\sigma$.  Suppose that $u \in C(\mathbb{R}^n \backslash \{0\})$ is a nonnegative solution of \eqref{Int=21xb0r1} with $p>0$, and $u$ has a non-removable singularity at the origin when $\alpha=0$ and $p=\frac{n+2\sigma}{n-2\sigma}$. 
\begin{itemize}
\item [$(1)$] If $0  <   p \leq \frac{n+2\sigma-2\alpha}{n-2\sigma}$,  then $u$ is radially symmetric and monotonically decreasing with respect to the origin.  

\item [$(2)$] If $\frac{n+2\sigma-2\alpha}{n-2\sigma} < p \leq \frac{n+2\sigma-\alpha}{n-2\sigma}$, then $u$ is radially symmetric with respect to the origin.  
\end{itemize} 
\end{theorem}  

\br\label{dkfi904}
When $\alpha=0$, Theorem \ref{THM=03} has been proved by Chen-Li-Ou \cite{CLO05} via the method of moving planes in an integral form. Here we will apply the method of moving spheres in  Li, Zhang and Zhu \cite{Li04,LZ03,Li-Zhu} to show Theorem \ref{THM=03} for all $\alpha \in [0, 2\sigma)$. 
\er

For $x\in \mathbb{R}^n, \lambda >0$ and a nonnegative  function $u$,  we define the Kelvin transform of $u$ with respect to the ball $B_\lambda(x)$ by
\begin{equation}\label{N0Ta-0022}
u_{x, \lambda}(\xi) = \left( \frac{\lambda}{|\xi - x|} \right)^{n-2\sigma} u(\xi^{x, \lambda}), 
\end{equation}
where 
\begin{equation}\label{N0Ta-0021}
\xi^{x, \lambda}= x+ \frac{\lambda^2 (\xi - x)}{|\xi - x|^2} ~~~~~  \textmd{for} ~  \xi\neq x.  
\end{equation}  
One can check that $(\xi^{x,\lambda})^{x, \lambda} =\xi$ and $(u_{x,\lambda})_{x,\lambda} \equiv u$.  If $x=0$, we use the notation $u_\lambda=u_{0,\lambda}$ and  $\xi^{\lambda}=\xi^{0, \lambda}$ as in Section \ref{S2=0hy}.   Making the following change of variables
$$
y=z^{x, \lambda} =x+ \frac{\lambda^2 (z - x)}{|z - x|^2}, 
$$
we have
$$
\aligned
\int_{|y -x| \geq \lambda} \frac{|y|^{-\alpha} u(y)^{p}}{|\xi^{x,\lambda} - y|^{n-2\sigma}} dy & = \int_{|z-x| \leq \lambda} \frac{|z^{x, \lambda}|^{-\alpha} u(z^{x, \lambda})^{p}}{|\xi^{x,\lambda} - z^{x, \lambda}|^{n-2\sigma}} \left( \frac{\lambda}{|z-x|} \right)^{2n} dz \\
& = \left( \frac{\lambda}{|\xi -x|} \right)^{-(n-2\sigma)} \int_{|z - x|\leq \lambda} \frac{|z^{x, \lambda}|^{-\alpha} u_{x,\lambda}(z)^p}{|\xi - z|^{n-2\sigma}} \left( \frac{\lambda}{|z-x|} \right)^{n+2\sigma - p(n-2\sigma)} dz, 
\endaligned
$$
where we have used the fact   
$$
 \frac{|z-x|}{\lambda} \frac{|\xi -x|}{\lambda} |\xi^{x,\lambda} - z^{x, \lambda}| = |\xi - z|.  
$$
Thus, we obtain 
\begin{equation}\label{Id-01}
\left( \frac{\lambda}{|\xi - x|} \right)^{n-2\sigma} \int_{|y - x|\geq \lambda}  \frac{|y|^{-\alpha} u(y)^{p}}{|\xi^{x,\lambda} - y|^{n-2\sigma}} dy  = \int_{|z - x|\leq \lambda} \frac{|z^{x, \lambda}|^{-\alpha} u_{x,\lambda}(z)^p}{|\xi - z|^{n-2\sigma}} \left( \frac{\lambda}{|z-x|} \right)^{n+2\sigma - p(n-2\sigma)} dz. 
\end{equation}
Similarly, we also have 
\begin{equation}\label{Id-02}
\left( \frac{\lambda}{|\xi - x|} \right)^{n-2\sigma} \int_{|y - x| \leq \lambda}  \frac{|y|^{-\alpha} u(y)^{p}}{|\xi^{x,\lambda} - y|^{n-2\sigma}} dy  = \int_{|z - x|\geq \lambda} \frac{|z^{x, \lambda}|^{-\alpha} u_{x,\lambda}(z)^p}{|\xi - z|^{n-2\sigma}} \left( \frac{\lambda}{|z-x|} \right)^{n+2\sigma - p(n-2\sigma)} dz.  \end{equation} 
Let $u \in C(\mathbb{R}^n \backslash \{0\})$ be a nonnegative solution of \eqref{Int=21xb0r1}. Then,  using \eqref{Id-01} and \eqref{Id-02}  we get 
\begin{equation}\label{ABC-01}
u_{x,\lambda}(\xi) = \int_{\mathbb{R}^n} \frac{|z^{x, \lambda}|^{-\alpha} u_{x,\lambda}(z)^p}{|\xi - z|^{n-2\sigma}} \left( \frac{\lambda}{|z-x|} \right)^{n+2\sigma - p(n-2\sigma)} dz 
\end{equation}
and
\begin{equation}\label{ABC-0hy2}
u(\xi) -u_{x,\lambda}(\xi) =\int_{|z-x|\geq \lambda} K(x, \lambda; \xi, z) \left[  \frac{u(z)^p}{|z|^{\alpha}}  - \left( \frac{\lambda}{|z-x|} \right)^{n+2\sigma - p(n-2\sigma)} \frac{u_{x,\lambda}(z)^p}{|z^{x,\lambda}|^{\alpha}}  \right] dz,   
\end{equation} 
where
$$
K(x, \lambda; \xi, z)=\frac{1}{|\xi -z|^{n-2\sigma}} - \left(\frac{\lambda}{|\xi - x|}  \right)^{n-2\sigma} \frac{1}{|\xi^{x,\lambda} -z|^{n-2\sigma}}. 
$$
It is elementary to check that   
$$
K(x, \lambda; \xi, z) >0~~~~~ \textmd{for} ~ \textmd{all} ~ |\xi -x|, |z-x| >\lambda>0.  
$$

We first prove the following result, which allows us to move the spheres from a starting point. 

\begin{lemma}\label{LeM=001}
Suppose that  $0< \sigma  < n/2$ and $-\infty <  \alpha < 2\sigma$.  Let $u \in C(\mathbb{R}^n \backslash \{0\})$ be a positive solution of \eqref{Int=21xb0r1} with  $p >0$.  Then for every $x \in \mathbb{R}^{n} \backslash \{0\}$, there exists a real number $\lambda_2 \in (0, |x|)$ such that for any $0< \lambda < \lambda_2$, we have
\begin{equation}\label{5SC101}
u_{x,\lambda} (y) \leq u(y) ~~~~~ \forall ~ |y-x| \geq \lambda,~ y \ne 0, 
\end{equation}
where $u_{x, \lambda}$ is the Kelvin transform  of $u$ as defined in \eqref{N0Ta-0022}.  
\end{lemma} 
\begin{proof}
The proof of Lemma \ref{LeM=001} consists of two steps. 

{\it Step 1. There exists $0< \lambda_1 < |x|$ such that for any $0< \lambda<\lambda_1$, 
\begin{equation}\label{5S0-0-1}
u_{x,\lambda} (y) \leq u(y)~~~~~~ \forall ~0< \lambda\leq  |y - x| \leq \lambda_1.
\end{equation} }
By  \cite[Theorem 2.5]{JLX17}, we have $u \in C^1({\mathbb{R}^n \backslash \{0\} })$.  Suppose that 
$$
|\nabla \ln u| \leq C_1 ~~~~~ \textmd{in} ~ B_{|x|/2}(x)
$$  
for some constant $C_1>0$.  Then  we have
\begin{equation}\label{Mono01}
\aligned
\frac{d}{dr} (r^{\frac{n-2\sigma}{2}} u(x + r\theta)) &= r^{\frac{n-2\sigma}{2} -1} u(x + r\theta) \left( \frac{n-2\sigma}{2} - r \frac{\nabla u \cdot \theta}{u}\right) \\
&\geq r^{\frac{n-2\sigma}{2} -1} u(x + r\theta) \left( \frac{n-2\sigma}{2} - C_1 r \right) >0
\endaligned
\end{equation}
for all $0 < r < \lambda_1 :=\min\{ \frac{n-2\sigma}{2C_1}, \frac{|x|}{2}\}$ and $\theta \in \mathbb{S}^{n-1}$. For any  $0< \lambda < |y - x| \leq \lambda_1$, let $\theta=\frac{y-x}{|y-x|}$, $r_1 = |y - x|$ and $r_2=\frac{\lambda^2 r_1}{|y - x|^2}$. Applying  \eqref{Mono01}  we obtain 
$$
r_2^{\frac{n-2\sigma}{2}} u(x + r_2\theta) < r_1^{\frac{n-2\sigma}{2}} u(x + r_1 \theta). 
$$
This is equivalent to 
$$
u_{x,\lambda} (y) \leq u(y), ~~~~~~0< \lambda\leq  |y - x| \leq \lambda_1.
$$

{\it Step 2.  There exists  $0< \lambda_2 <  \lambda_1  <|x| $ such that \eqref{5SC101} holds for all $0< \lambda < \lambda_2$. } By Fatou lemma we get 
$$
 \liminf_{ |x| \to \infty }   |x|^{n-2\sigma} u(x) = \liminf_{ |x| \to \infty}   \int_{\mathbb{R}^n} \frac{|x|^{n-2\sigma} u(y)^{p} }{|y|^{\alpha} |x - y|^{n-2\sigma}} dy  \geq \int_{\mathbb{R}^n} \frac{u(y)^{p}}{|y|^\alpha} dy >0. 
$$
Hence,   there exist two constants $c_0, R_0>0$ such that
\begin{equation}\label{5SC102}
u(y) \geq \frac{c_0}{|y|^{n-2\sigma}}~~~~~~~ \textmd{for} ~  \textmd{all} ~  |y|\geq R_0. 
\end{equation}
For $y \in B_{R_0} \backslash \{0\}$, by the equation \eqref{Int=21xb0r1} and positivity of $u$ we have 
$$
u(y) \geq \int_{B_{R_0}} \frac{u(z)^p }{|y|^\alpha |y - z|^{n-2\sigma}} dz \geq (2R_0)^{2\sigma -n} \int_{B_{R_0}} \frac{u(z)^{p}}{|y|^\alpha} dz  >0.
$$
This, together with \eqref{5SC102}, implies that there exists $c_1>0$ such that 
$$
u(y) \geq \frac{c_1}{|y - x|^{n-2\sigma}} ~~~~~ \forall ~ |y - x| \geq \lambda_1, ~ y \ne 0. 
$$
Thus, for sufficiently small  $\lambda_2 \in (0, \lambda_1 )$ and for any $0< \lambda< \lambda_2$, we have 
$$
\aligned
u_{x,\lambda}(y) & = \left( \frac{\lambda}{|y - x|} \right)^{n-2\sigma} u \left( x + \frac{\lambda^2(y - x)}{ |y- x|^2 } \right) \\
& \leq \left( \frac{\lambda_2}{|y - x|} \right)^{n-2\sigma} \sup_{B_{\lambda_1}(x)} u \leq u(y), ~~~ \forall ~ |y-x| \geq \lambda_1, ~ y \ne 0.  
\endaligned 
$$
From \eqref{5S0-0-1} and the above estimate,  the proof of Lemma \ref{LeM=001} is completed.   
\end{proof}

For any $x \in \mathbb{R}^n \backslash \{0\}$,  we define 
$$
\bar{\lambda}(x):= \sup\{ 0< \mu \leq |x|~ | ~ u_{x,\lambda}(y) \leq u(y),~ \forall ~ |y - x| \geq \lambda,~ y \ne 0, ~ \forall ~ 0< \lambda<\mu \}. 
$$
By Lemma \ref{LeM=001}, $\bar{\lambda}(x)$ is well defined and $\bar{\lambda}(x)>0$.  Next we will show that $\bar{\lambda}(x) =|x|$ for all $x \in \mathbb{R}^{n} \backslash \{0\}$.  Before proving this,  we need the following lemma which is similar to \cite[Lemma 4]{CaLi} and \cite[Lemma 2.1]{Jin}. 
We include its proof here for completeness.    
\begin{lemma}\label{S67_Le=001} 
For $\lambda \in (0, |x|)$,  we have 
\begin{equation}\label{Syh_Le=00234}
\frac{\lambda^2 |z|}{|z-x|^2  \Big| x +  \frac{\lambda^2(z-x)}{|z-x|^2} \Big|} < 1
\end{equation}
for any $z \in \mathbb{R}^n$ satisfying $|z -x| >\lambda$.   
\end{lemma} 

\begin{proof} 
For any $z \in \mathbb{R}^n$ with $|z -x| >\lambda$, we have  that 
$$
\aligned
\eqref{Syh_Le=00234} & \Leftrightarrow  \lambda^2 |z| < |z-x|^2  \Big| x +  \frac{\lambda^2(z-x)}{|z-x|^2} \Big|  \\
& \Leftrightarrow  \lambda^2 |z| < \big|  (|z-x|^2  - \lambda^2) x + \lambda^2 z  \big| \\
& \Leftrightarrow  -2\lambda^2 \langle  z, x \rangle  <  (|z-x|^2 - \lambda^2) |x|^2 ~~~ (\textmd{by} ~\textmd{taking}~  \textmd{square}) \\ 
& \Leftrightarrow  -2\lambda^2 \langle  (z - x), x \rangle  <  (|z-x|^2 + \lambda^2) |x|^2. 
\endaligned
$$
The last inequality holds since  
$$
-2\lambda^2 \langle  (z - x), x \rangle \leq 2 \lambda^2|z - x| |x| < 2 \lambda |z-x| |x|^2 < (|z-x|^2 + \lambda^2) |x|^2. 
$$
Lemma \ref{S67_Le=001} is established.  
\end{proof}  

\begin{lemma}\label{LeM=002}
Suppose that  $0< \sigma  < n/2$ and $0 \leq  \alpha < 2\sigma$.  Let $u \in C(\mathbb{R}^n \backslash \{0\})$ be a positive solution of \eqref{Int=21xb0r1} with  $0 < p \leq \frac{n+2\sigma-2\alpha}{n-2\sigma}$.  Then $\bar{\lambda}(x) =|x|$ for all $x \in \mathbb{R}^{n} \backslash \{0\}$.
\end{lemma} 
\begin{proof}
Suppose $\bar{\lambda}(x) < |x|$ for some $x \in \mathbb{R}^{n} \backslash \{0\}$. For simplicity we write $\bar{\lambda}=\bar{\lambda}(x) $ and denote $ \bar{\delta}:=\min \{ 1, \frac{|x| - \bar{\lambda}}{2} \} >0$.  
By the definition of $\bar{\lambda}$, 
\begin{equation}\label{5CL6-01}
u_{x, \bar{\lambda} } (y) \leq u(y) ~~~~~~ \textmd{for} ~ \textmd{all} ~ |y - x|\geq \bar{\lambda}, ~ y \ne 0. 
\end{equation} 

We next will separate into three cases to show that $u_{x, \bar{\lambda} } (y) < u(y)$ for any $|y - x|\geq \bar{\lambda}$ with  $y \ne 0$, and there exists $\varepsilon_1 \in (0, 1)$ such that 
\begin{equation}\label{Lhhdy6-03=001}
(u - u_{x, \bar{\lambda}}) (y) \geq \frac{\varepsilon_1}{|y -x|^{n-2\sigma}} ~~~~~~  \textmd{for} ~ \textmd{all} ~  |y -x |\geq \bar{\lambda} +\bar{\delta}, ~ y \ne 0. 
\end{equation}

 {\it Case $1$}:  $0 < \alpha < 2\sigma$.  
Since $n+2\sigma - p(n-2\sigma) \geq 2\alpha$, we obtain that  for $|z - x|\geq \lambda >0$, 
$$
\left( \frac{\lambda}{|z- x |} \right)^{n+2\sigma - p(n-2\sigma)} \leq \left( \frac{\lambda}{|z- x |} \right)^{2\alpha}. 
$$
Therefore,   using \eqref{ABC-0hy2}, \eqref{5CL6-01},  Lemma \ref{S67_Le=001} and the positivity of the kernel $K$, we have, for any $|y -x|\geq \bar{\lambda}$ with $y \ne 0$,  
$$
\aligned
u(y) -u_{x,\bar{\lambda}}(y) & =\int_{|z-x|\geq \bar{\lambda}} K(x, \bar{\lambda}; y, z) \left[  \frac{u(z)^p}{|z|^{\alpha}}  - \left( \frac{\bar{\lambda}}{|z-x|} \right)^{n+2\sigma - p(n-2\sigma)} \frac{u_{x,\bar{\lambda}}(z)^p}{|z^{x,\bar{\lambda}}|^{\alpha}}  \right] dz \\
& \geq \int_{|z-x|\geq \bar{\lambda}} K(x, \bar{\lambda}; y, z) \left[  \frac{u(z)^p}{|z|^{\alpha}}  - \left( \frac{\bar{\lambda}}{|z-x|} \right)^{2\alpha} \frac{u_{x,\bar{\lambda}}(z)^p}{|z^{x,\bar{\lambda}}|^{\alpha}}  \right] dz \\
& \geq \int_{|z-x|\geq \bar{\lambda}} K(x, \bar{\lambda}; y, z) \Bigg[ 1  - \Bigg(  \frac{\bar{\lambda}^2 |z|}{|z-x|^2 \big| x + \frac{\bar{\lambda}^2 (z-x)}{|z-x|^2} \big|} \Bigg)^{\alpha} \Bigg]  \frac{u_{x,\bar{\lambda}}(z)^p}{|z|^{\alpha}}  dz \\
& >0. 
\endaligned
$$ 
This, together with Fatou Lemma, yields 
$$
\aligned
 &\liminf_{ |y| \to \infty}   |y-x|^{n-2\sigma} ( u -u_{x,\bar{\lambda}} )(y) \\
 & ~~ \geq   \liminf_{ |y| \to \infty}   
 \int_{|z -x|\geq \bar{\lambda}} |y-x|^{n-2\sigma}  K(x, \bar{\lambda}; y, z)  \Bigg[ 1  - \Bigg(  \frac{\bar{\lambda}^2 |z|}{|z-x|^2 \big| x + \frac{\bar{\lambda}^2 (z-x)}{|z-x|^2} \big|} \Bigg)^{\alpha} \Bigg]  \frac{u_{x,\bar{\lambda}}(z)^p}{|z|^{\alpha}}  dz  \\
 &~~ \geq \int_{|z -x|\geq \bar{\lambda}} \Bigg[ 1 - \left( \frac{\bar{\lambda}}{|z - x |} \right)^{n - 2\sigma} \Bigg] \Bigg[ 1  - \Bigg(  \frac{\bar{\lambda}^2 |z|}{|z-x|^2 \big| x + \frac{\bar{\lambda}^2 (z-x)}{|z-x|^2} \big|} \Bigg)^{\alpha} \Bigg]  \frac{u_{x,\bar{\lambda}}(z)^p}{|z|^{\alpha}}  dz >0. 
\endaligned
$$
Consequently, there exist $c_2>0$ and sufficiently large $R_2 >0$  such that 
\begin{equation}\label{5CL6-02}
( u -u_{x,\bar{\lambda}} )(y)  \geq \frac{c_2}{|y -x|^{n-2\sigma}}~~~~~~  \textmd{for} ~ \textmd{all} ~  |y -x |\geq R_2. 
\end{equation} 
On the other hand,  by the positivity of the kernel $K$ there exists $\beta_1>0$ such that for all $\bar{\lambda} + \bar{\delta} \leq |y -x| \leq R_2$ and $2R_2  \leq |z -x| \leq 4R_2$,   
\be\label{Del-33K=01}
K(x, \bar{\lambda}; y, z) \geq  \beta_1 >0. 
\ee
Thus,  for $\bar{\lambda} + \bar{\delta} \leq |y -x| \leq R_2$ and  $y \ne 0$,  we have
$$
\aligned
u(y) - u_{x,\bar{\lambda}}(y)  &\geq  \int_{|z-x|\geq \bar{\lambda}} K(x, \bar{\lambda}; y, z) \Bigg[ 1  - \Bigg(  \frac{\bar{\lambda}^2 |z|}{|z-x|^2 \big| x + \frac{\bar{\lambda}^2 (z-x)}{|z-x|^2} \big|} \Bigg)^{\alpha} \Bigg]  \frac{u_{x,\bar{\lambda}}(z)^p}{|z|^{\alpha}}  dz \\
& \geq  \int_{ 2R_2  \leq |z -x| \leq 4R_2 }  \beta_1 \Bigg[ 1  - \Bigg(  \frac{\bar{\lambda}^2 |z|}{|z-x|^2 \big| x + \frac{\bar{\lambda}^2 (z-x)}{|z-x|^2} \big|} \Bigg)^{\alpha} \Bigg]  \frac{u_{x,\bar{\lambda}}(z)^p}{|z|^{\alpha}}   dz \\
& =:C_2 >0. 
\endaligned
$$
 Combining this with \eqref{5CL6-02},  there exists $\varepsilon_1 \in (0, 1)$ such that
$$
(u - u_{x, \bar{\lambda}}) (y) \geq \frac{\varepsilon_1}{|y -x|^{n-2\sigma}} ~~~~~~  \textmd{for} ~ \textmd{all} ~  |y -x |\geq \bar{\lambda} +\bar{\delta}, ~ y \ne 0. 
$$

{\it Case $2$}:  $\alpha =0$ {\it and}  $0< p < \frac{n+2\sigma}{n-2\sigma}$.   Using \eqref{ABC-0hy2}, \eqref{5CL6-01} and the positivity of the kernel $K$, we have, for any $|y -x|\geq \bar{\lambda}$ with $y \ne 0$,  
$$
\aligned
u(y) -u_{x,\bar{\lambda}}(y) & =\int_{|z-x|\geq \bar{\lambda}} K(x, \bar{\lambda}; y, z) \left[  u(z)^p  - \left( \frac{\bar{\lambda}}{|z-x|} \right)^{n+2\sigma - p(n-2\sigma)} u_{x,\bar{\lambda}}(z)^p \right] dz \\
& \geq \int_{|z-x|\geq \bar{\lambda}} K(x, \bar{\lambda}; y, z) \left[ 1  - \left( \frac{\bar{\lambda}}{|z-x|} \right)^{n+2\sigma - p(n-2\sigma)} \right]   u_{x,\bar{\lambda}}(z)^p  dz \\
& >0. 
\endaligned
$$
By a similar argument as in Case 1, we obtain that \eqref{Lhhdy6-03=001} holds.   

{\it Case $3$}:  $\alpha =0$,  $p = \frac{n+2\sigma}{n-2\sigma}$  {\it and  $0$ is a non-removable singularity}.  In this case, we have $u_{x, \bar{\lambda}} (y) \not\equiv u(y)$.  By \eqref{ABC-0hy2}, \eqref{5CL6-01} and the positivity of the kernel $K$, we have, for any $|y -x|\geq \bar{\lambda}$ with $y \ne 0$,  
$$
u(y) -u_{x,\bar{\lambda}}(y) = \int_{|z-x|\geq \bar{\lambda}} K(x, \bar{\lambda}; y, z) \left[  u(z)^{\frac{n+2\sigma}{n-2\sigma}}  -  u_{x,\bar{\lambda}}(z)^{\frac{n+2\sigma}{n-2\sigma}} \right] dz >0.  
$$
By Fatou lemma, we have 
$$
\aligned
& \liminf_{ |y| \to \infty}    |y-x|^{n-2\sigma} ( u -u_{x,\bar{\lambda}} )(y) \\
& ~~ \geq \int_{|z -x|\geq \bar{\lambda}} \left[ 1 - \left( \frac{\bar{\lambda}}{|z - x |} \right)^{n - 2\sigma} \right] \left( u(z)^{\frac{n+2\sigma}{n-2\sigma}} - u_{x, \bar{\lambda}}(z)^{\frac{n+2\sigma}{n-2\sigma}} \right) dz >0. 
\endaligned
$$
An argument similar to Case 1 leads to that \eqref{Lhhdy6-03=001} holds.  

Therefore,  in any case, we have that \eqref{Lhhdy6-03=001} holds.   By \eqref{Lhhdy6-03=001} and the explicit formula of $u_{x, \lambda}$,  there exists a small $\varepsilon_2 \in (0,  \bar{\delta}/2)$ such that for any  $\bar{\lambda} \leq \lambda \leq \bar{\lambda} + \varepsilon_2 < |x|$,  
\begin{equation}\label{5CL6-04}
\aligned
(u - u_{x, \lambda}) (y) & \geq \frac{\varepsilon_1}{|y -x|^{n-2\sigma}} + ( u_{x, \bar{\lambda}} - u_{x, \lambda}) (y) \\
 & \geq  \frac{\varepsilon_1}{2|y -x|^{n-2\sigma}}  ~~~~~~ \forall ~  |y -x |\geq \bar{\lambda} + \bar{\delta},  ~ y \ne  0. 
\endaligned
\end{equation} 
For $\varepsilon \in (0, \varepsilon_2)$ which we choose below, using \eqref{ABC-0hy2} and Lemma \ref{S67_Le=001}, we have, for $\bar{\lambda} \leq \lambda \leq \bar{\lambda} + \varepsilon$ and $\lambda \leq |y -x| \leq \bar{\lambda} + \bar{\delta}$, 
$$
\aligned
u(y) - u_{x,\lambda}(y) 
& =\int_{|z-x|\geq \lambda} K(x, \lambda; y, z) \left[  \frac{u(z)^p}{|z|^{\alpha}}  - \left( \frac{\lambda}{|z-x|} \right)^{n+2\sigma - p(n-2\sigma)} \frac{u_{x,\lambda}(z)^p}{|z^{x,\lambda}|^{\alpha}}  \right] dz \\ 
&  \geq \int_{|z-x|\geq \lambda} K(x, \lambda; y, z) \Bigg[  \frac{u(z)^p}{|z|^{\alpha}}   - \Bigg(  \frac{\lambda^2 |z|}{|z-x|^2 \big| x + \frac{\lambda^2 (z-x)}{|z-x|^2} \big|} \Bigg)^{\alpha} \frac{u_{x,\lambda}(z)^p}{|z|^{\alpha}} \Bigg]    dz \\
&  \geq \int_{|z-x|\geq \lambda} K(x, \lambda; y, z)  \frac{u(z)^p - u_{x,\lambda}(z)^p}{|z|^{\alpha}}  dz,  \\
\endaligned
$$
and thus,  by  \eqref{5CL6-04} and  \eqref{5CL6-01}, 
$$
\aligned
u(y) - u_{x,\lambda}(y)  &  \geq  \int_{\lambda \leq |z -x| \leq \bar{\lambda} +\bar{\delta}}  K(x, \lambda; y, z)  \frac{u(z)^p - u_{x,\lambda}(z)^p}{|z|^{\alpha}}  dz \\
&~~~~  + \int_{\bar{\lambda} + 2 \leq |z -x| \leq \bar{\lambda} +3}  K(x, \lambda; y, z)  \frac{u(z)^p - u_{x,\lambda}(z)^p}{|z|^{\alpha}} dz \\
& \geq -C \int_{\lambda \leq |z -x| \leq \lambda +\varepsilon}  K(x,\lambda; y, z) \frac{|z -x| -\lambda}{|z|^\alpha} dz \\ 
&~~~~ + \int_{\lambda+\varepsilon <  |z -x| \leq \bar{\lambda} +\bar{\delta} }  K(x,\lambda; y, z)   \frac{u_{x,\bar{\lambda}}(z)^{p} - u_{x, \lambda}(z)^{p}}{|z|^\alpha}   dz \\
&~~~~ + \int_{\bar{\lambda} + 2 \leq |z -x| \leq \bar{\lambda} +3}  K(x,\lambda; y, z) \frac{ u(z)^{p} - u_{x,\lambda}(z)^{p}}{|z|^\alpha} dz,
\endaligned
$$
where we have used
$$
| u(z)^{p} - u_{x,\lambda}(z)^{p} |\leq C (|z -x| -\lambda) 
$$
in the second inequality. By \eqref{5CL6-04}, there exists $\delta_1 >0$ such that for any  $\bar{\lambda} \leq  \lambda \leq \bar{\lambda} +\varepsilon$, 
$$
u(z)^{p} - u_{x,\lambda}(z)^{p} > \delta_1 ~~~~~~\forall ~\bar{\lambda} +2 \leq |z-x| \leq \bar{\lambda} +3,~ z \ne 0.
 $$
Since $u \in C^1(B_{\bar{\lambda} + \varepsilon_2}(x))$, there exists $C>0$ (independent of $\varepsilon$) such that for all $\bar{\lambda} \leq  \lambda \leq \bar{\lambda} +\varepsilon$,
 $$
 | u_{x, \bar{\lambda}}(z)^{p} - u_{x,\lambda}(z)^{p} | \leq C (\lambda - \bar{\lambda}) \leq C\varepsilon ~~~~~~\forall ~ \bar{\lambda} \leq  \lambda  \leq |z-x| \leq \bar{\lambda} + \bar{\delta}. 
 $$
Moreover, we have, for $\bar{\lambda} \leq  \lambda \leq \bar{\lambda} +\varepsilon$ and for $\lambda \leq |y -x| \leq \bar{\lambda}+\bar{\delta}$, 
$$
\aligned
\int_{\lambda+\varepsilon  \leq |z-x| \leq \bar{\lambda} +\bar{\delta}}  \frac{K(x,\lambda; y, z)}{|z|^\alpha}  dz &  \leq C \left|  \int_{\lambda+\varepsilon \leq |z -x| \leq \bar{\lambda} +\bar{\delta}}  \left( \frac{1}{|y -z|^{n-2\sigma}} -  \frac{1}{|y^{x,\lambda} -z|^{n-2\sigma}}\right) dz  \right| \\
& ~~~~ + C\int_{\lambda+\varepsilon \leq |z-x| \leq \bar{\lambda} +\bar{\delta}}  \left|  \left(\frac{\lambda}{|y-x|}\right)^{n-2\sigma} -1 \right|   \frac{1}{|y^{x,\lambda} -z|^{n-2\sigma}} dz\\
& \leq C (\varepsilon^{2\sigma-1}  +  |\ln \varepsilon|  +1) (|y-x| - \lambda) 
\endaligned
$$
and 
$$
\aligned
\int_{ \lambda \leq |z-x| \leq \lambda +\varepsilon }   K(x,\lambda; y, z) \frac{|z-x| -\lambda}{|z|^\alpha} dz & \leq C\left|  \int_{ \lambda \leq |z-x| \leq \lambda +\varepsilon }   \left( \frac{|z-x| -\lambda}{|y -z|^{n-2\sigma}} -  \frac{|z-x| -\lambda}{|y^{x,\lambda} -z|^{n-2\sigma}}\right) dz  \right| \\
& ~~~ + \varepsilon C\int_{ \lambda \leq |z-x| \leq \lambda +\varepsilon}  \left|  \left(\frac{\lambda}{|y-x|}\right)^{n-2\sigma} -1 \right|   \frac{1}{|y^{x,\lambda} -z|^{n-2\sigma}} dz\\
& \leq C (|y-x| - \lambda) \varepsilon^{2\sigma/n} + C\varepsilon(|y-x| - \lambda) \\
& \leq C (|y-x| - \lambda) \varepsilon^{2\sigma/n}. 
\endaligned
$$
On the other hand, since $K(x, \lambda; y ,z) =0$ for $y \in \partial B_\lambda(x)$ and 
$$
\langle \nabla_y K(x, \lambda; y, z), y-x \rangle \big|_{|y-x|=\lambda} = (n-2\sigma) |y -z|^{2\sigma-n-2} (|z - x|^2 - |y -x|^2) >0  
$$
for all  $\bar{\lambda} + 2 \leq |z-x| \leq \bar{\lambda} +3$,   and using the positivity of the kernel $K$,  we obtain that,  for $\bar{\lambda} \leq  \lambda \leq |y -x| \leq \bar{\lambda}+ \bar{\delta}$,    
$$
K(x, \lambda; y ,z) \geq \delta_2 (|y-x| - \lambda) ~~~~~~~ \forall ~ \bar{\lambda} + 2 \leq |z-x| \leq \bar{\lambda} +3,
$$
where $\delta_2>0$ is independent of $\varepsilon$. Thus, it follows from the above that for $\bar{\lambda} \leq \lambda \leq \bar{\lambda} + \varepsilon$ and for $\lambda \leq |y -x| \leq \bar{\lambda} + \bar{\delta}$,  
$$
\aligned
u(y) - u_{x,\lambda}(y)  \geq \left( -C \varepsilon^{\frac{2\sigma}{n}}  + \delta_1 \delta_2 \int_{\bar{\lambda} + 2 \leq |z -x| \leq \bar{\lambda} +3}  \frac{ 1 }{|z|^\alpha} dz \right) ( |y -x| -\lambda) \geq 0
\endaligned
$$
if $\varepsilon$ is chosen sufficiently small.  This and \eqref{5CL6-04} contradict the definition of $\bar{\lambda}$.  Lemma \ref{LeM=002} is established.  
\end{proof}

\noindent{\it Proof of Theorem \ref{THM=03}}.  
We first assume that $0 < p \leq \frac{n+2\sigma-2\alpha}{n-2\sigma}$.  Let $u \in C(\mathbb{R}^n \backslash \{0\})$ be a nonnegative solution of \eqref{Int=21xb0r1}.  If $u(x_0) =0 $ for some $x_0 \in \mathbb{R}^n \backslash \{0\}$, then 
$$
\int_{\mathbb{R}^n} \frac{u(y)^{p}}{|y|^{\alpha} |x_0-y|^{n-2\sigma}} dy =0, 
$$
which implies that $u \equiv 0$ on $\mathbb{R}^n \backslash \{0\}$.  Next we suppose that $u \in C(\mathbb{R}^n \backslash \{0\})$ be a positive solution of \eqref{Int=21xb0r1}. It follows from Lemma \ref{LeM=002} that for every $x\in \mathbb{R}^n \backslash \{0\}$, 
\begin{equation}\label{SS555}
u_{x,\lambda}(y) \leq u(y)~~~~~~ \forall ~ |y - x| \geq \lambda, ~y \ne 0, ~~\forall ~ 0< \lambda< |x|. 
\end{equation}
For  any unit vector $e \in \mathbb{R}^{n}$,  for any $a >0$, for any $ z\in \mathbb{R}^{n}\backslash \{0\}$  satisfying $(z - ae ) \cdot e <0$, and for any $R >a$, we have, by \eqref{SS555} with $x=Re$ and $\lambda = R - a$, 
$$
u(z)\geq u_{x,\lambda}(z)=\left( \frac{\lambda}{|z - x|}\right)^{n-2\sigma} u\left( x + \frac{\lambda^2(z -x)}{|z -x|^2} \right).
$$
Sending $R$ to infinity in the above,  we obtain 
\begin{equation}\label{SS666}
u(z) \geq u(z - 2( z\cdot e  - a )e ).
\end{equation}
Since $e \in\mathbb{R}^{n}$ and $a>0$ are arbitrary, \eqref{SS666} implies that $u$ is radially symmetric about the origin. Moreover,  \eqref{SS666} also gives
$$
u(z)=u(z_1, z_2, \dots, z_n) \geq u_a(z):=u(2a - z_1, z_2, \dots, z_n)~~~~~ \forall~ z_1\leq a,~ a>0.
$$
Thus,   $u$ is also monotonically decreasing about the origin.

Next we consider the case $\frac{n+2\sigma-2\alpha}{n-2\sigma} < p \leq \frac{n+2\sigma -\alpha}{n-2\sigma}$ and $0 < \alpha <2\sigma$. Let  $u \in C(\mathbb{R}^n \backslash \{0\})$ be a positive solution of \eqref{Int=21xb0r1}. Define 
$$
v(x) =\left( \frac{1}{|x|} \right)^{n-2\sigma} u\left( \frac{x}{|x|^2} \right) ~~~~~ \textmd{for}  ~ x\ne 0.  
$$
By \eqref{ABC-01}, $v$ is a positive solution to the following integral equation    
$$
v(x)= \int_{\mathbb{R}^n} \frac{v(y)^p}{|y|^{\theta} |x - y|^{n-2\sigma}} dy  
$$
with $\theta:= n+2\sigma -\alpha - p(n-2\sigma)$. Note that the new exponent $\theta$ satisfies 
$$
\aligned
0\leq \theta  & \iff p\leq \frac{n+2\sigma -\alpha}{n-2\sigma}, \\
\theta < \alpha & \iff \frac{n+2\sigma -2\alpha}{n-2\sigma} <p ~~ (\textmd{hence} ~ \theta < 2\sigma), \\
p < \frac{n+2\sigma -2\theta}{n-2\sigma} & \iff  p >  \frac{n+2\sigma -2\alpha}{n-2\sigma}. 
\endaligned 
$$
From the previous case, we know that $v$ is radially symmetric about the origin, and hence, $u$ is also radially symmetric about the origin.   Theorem \ref{THM=03}  is established.   
\hfill$\square$

\vskip0.10in

\noindent{\it Proof of Theorem \ref{THM=01}. } It follows from Theorems \ref{THM=02} and \ref{THM=03}.   
\hfill$\square$

\end{document}